\documentclass[11pt,reqno]{amsart}
\usepackage {amsmath, amsfonts, amssymb, amscd, amsthm, eucal}

\usepackage{amsbsy}
\usepackage[all]{xy}
\SelectTips{cm}{}
\SilentMatrices
\usepackage[dvipsnames,usenames]{color}

\theoremstyle{plain}{
  \newtheorem{thm}{Theorem}[section]
  \newtheorem{cor}[thm]{Corollary}
  \newtheorem{lem}[thm]{Lemma}
  \newtheorem{prop}[thm]{Proposition}

  \newtheorem{sssection}[thm]{}}

\theoremstyle{definition}{
  \newtheorem{defn}[thm]{Definition}
  
  }

\theoremstyle{remark}{
  \newtheorem{rem}[thm]{Remark}}

\renewcommand{\subsubsection}{\sssection\rm}

\numberwithin{equation}{section}

\newcommand{\pt}{pt}

\newcommand{\id}{\mathrm{id}}

\newcommand{\Hom}{\operatorname{Hom}}

\DeclareMathOperator{\rk}{rk}

\DeclareMathOperator{\Th}{Th}
\DeclareMathOperator{\MGL}{\mathbf {MGL}}
\newcommand{\SmOp}{\mathcal Sm\mathcal Op}
\newcommand{\Sm}{\mathcal Sm}

\newcommand{\Aff}{\mathbf{A}}

\newcommand{\Gr}{Gr}

\newcommand{\HP}{HP}

\newcommand{\ZZ}{\mathbb{Z}}

\newcommand{\M}{\mathbf{M}}

\newcommand{\colim}{\operatornamewithlimits{colim}}

\newcommand \xra {\xrightarrow }

\newcommand \lra {\longrightarrow }
\newcommand \hra {\hookrightarrow }

\newcommand \xla {\xleftarrow }

\newcommand{\MSp}{\mathbf{MSp}}
\newcommand{\MSL}{\mathbf{MSL}}
\newcommand{\angles}[1]{\langle #1 \rangle}
\DeclareMathOperator{\HGr}{\mathnormal{HGr}}

\newcommand{\OO}{\mathcal{O}}
\newcommand{\parens}[1]{\textup{(}#1\textup{)}}

\DeclareMathOperator{\thom}{\mathnormal{th}}
\newcommand{\onto}{\twoheadrightarrow}

\newcommand{\sym}{\Sigma}
\DeclareMathOperator{\hocolim}{hocolim}
\newcommand{\finite}{\text{\textit{fin}}}

\newcommand{\shf}{\mathcal}

\newcommand{\homog}{\text{\textit{hom}}}
\newcommand{\sphere}{\mathbb{S}}

\newcommand{\GrSL}{SGr}
\newcommand{\NN}{\mathbb{N}}

\newcommand{\GG}{\mathbb{G}}

\newcommand{\TGL}[1]{\mathcal{T}GL_{#1}}
\newcommand{\TSL}[1]{\mathcal{T}SL_{#1}}
\newcommand{\TSp}[1]{\mathcal{T}Sp_{#1}}
\newcommand{\Ev}{\mathrm{Ev}}

\begin{document}

\title{On the algebraic cobordism spectra $\MSL$ and $\MSp$}

\author{Ivan Panin}
\address{Steklov Institute of Mathematics at St.~Petersburg, Russia}
\author {Charles Walter}
\address{Laboratoire J.-A.\ Dieudonn\'e (UMR 6621 du CNRS)\\
 D\'epartement de math\'ematiques\\
 Universit\'e de Nice -- Sophia Antipolis\\ 06108 Nice Cedex 02\\
 France}

\thanks{The first author gratefully acknowledge excellent working conditions and support provided by
Laboratoire J.-A. Dieudonn\'{e}, UMR 6621 du CNRS, Universite de Nice - Sophia-Antipolis,
and by the RCN Frontier Research Group Project no. 250399 “Motivic Hopf equations" at University of Oslo. 
}

%

\begin{abstract}
We construct algebraic cobordism spectra $\MSL$ and $\MSp$.
They are commutative monoids in the
category of symmetric $T^{\wedge 2}$-spectra.
The spectrum $\MSp$ comes with a natural symplectic orientation given either
by a tautological Thom class
$th^{\MSp} \in \MSp^{4,2}(\MSp_2)$, a tautological Borel class
$b_{1}^{\MSp} \in \MSp^{4,2}(HP^{\infty})$ or any of six other equivalent structures.
For a  commutative monoid $E$ in the category
${SH}(S)$
we prove that assignment
$\varphi \mapsto \varphi(\thom^{\MSp})$
identifies the set of homomorphisms of monoids
$\varphi\colon \MSp \to E$
in the motivic stable homotopy category
$SH(S)$
with the set of tautological Thom elements of symplectic orientations of $E$.
A weaker universality result is obtained for $\MSL$ and
special linear orientations.
\end{abstract}

\maketitle

\section{Introduction}

A dozen years ago Voevodsky \cite{Voevodsky:1998kx}
constructed the algebraic cobordism spectrum $\MGL$
in the motivic stable homotopy category $SH(S)$.  This gave a new cohomology theory
$\MGL^{*,*}$ on smooth schemes and on motivic spaces.  Later Vezzosi \cite{Vezzosi:2001aa}
put a commutative monoid structure on $\MGL$.  This gave a product to $\MGL^{*,*}$.
The commutative monoid structure can even be constructed in the symmetric monoidal model category
of symmetric $T$-spectra, with $T = \Aff^{1}/(\Aff^{1}-0)$ the Morel-Voevodsky
object  (Panin, Pimenov and R\"ondigs \cite{Panin:2008db}).

In this paper we construct the algebraic special linear and symplectic cobordism
spectra $\MSL$ and $\MSp$.   The construction of $\MSL$ is straightforward although there is one slightly
subtle point.  We equip each space $BSL_{n}$ and $\MSL_{n}$ with an action of $GL_{n}$ which
is compatible with the monoid structure $BSL_{m} \times BSL_{n} \to BSL_{m+n}$ induced by the
direct sum of subbundles.  This gives an action of the subgroup $\sym_{n} \subset GL_{n}$
of permutation matrices.  But to define the unit of the monoid structure we need the action
on $BSL_{n}$ to have fixed points.  The natural action of $SL_{n}$ has fixed points, but the natural action
of $GL_{n}$ does not.  So we use an embedding $\sym_{n} \subset Sp_{2n} \subset SL_{2n}$.
This means that our $\MSL$ is a commutative monoid in the category of
symmetric $T^{\wedge 2}$-spectra.
The categories of symmetric $T$-spectra and of symmetric $T^{\wedge 2}$-spectra
are both symmetrical monoidal, and their homotopy categories are equivalent symmetric
monoidal categories  (Theorem \ref{T:comparison}).  So a symmetric $T^{\wedge 2}$-spectrum
structure is quite satisfactory, and it seems to be a natural structure for this spectrum.

Cobordism spectra and the cohomology theories they define are expected to have some universal
properties among certain classes of cohomology theories.  For instance Voevodsky's and
Levine and Morel's algebraic cobordism theories  are universal
among oriented cohomology theories \cite{Levine:2007ys,Panin:2008db,Vezzosi:2001aa}.
We should therefore expect $\MSL$ to have
some degree of universality for special linearly oriented theories.
Recall that a special linear bundle $(E,\lambda)$ over $X$ is a pair consisting of a vector bundle
$E$ and an isomorphism of line bundles $\lambda \colon \OO_{X} \cong \det E$.  A special linear
orientation on a cohomology theory $A^{*,*}$ is an assignment to every special linear bundle of
a Thom class
$\thom(E,\lambda) \in A^{2n,n}(E,E-X) = A^{2n,n}_{X}(E)$ with $n = \rk E$ which is functorial,
multiplicative, and
such that the multiplication maps ${-}\cup \thom(E,\lambda) \colon A^{*,*}(X) \to A^{*+2n,*+n}(E,E-X)$
are isomorphisms. In the motivic context we generally also require
that the Thom class of the trivial line bundle over a point be
$\Sigma_{T}1_{A} \in A^{2,1}(T) = A^{2,1}(\Aff^{1},\Aff^{1}-0)$.  Hermitian $K$-theory and
Balmer's derived Witt groups are examples of special linearly oriented theories which are not
oriented.

The universality properties we show for $\MSL$ are as follows.
%
A morphism of commutative monoids $\varphi \colon (\MSL,\mu^{SL},e^{SL}) \to (A,\mu,e)$
in $SH(S)$ determines naturally a special linear orientation on $A^{*,*}$ with Thom classes
written $\thom^{\varphi}(E,\lambda)$.
The compatibility of $\varphi$ with the monoid structure ensures the multiplicativity of the
Thom classes (Theorem \ref{T:thom.phi}).


Conversely, a special linear orientation on $A^{*,*}$ with Thom classes $\thom(E,\lambda)$ determines a
morphism $\varphi \colon \MSL \to A$ in $SH(S)$ with
$\thom^{\varphi}(E,\lambda) = \thom(E,\lambda)$ for all $(E,\lambda)$.  This $\varphi$ is
unique modulo a certain subgroup $\varprojlim^{1} A^{2n-1,n}(MSL_{n}^{\angles{n}}) \subset
Hom_{SH(S)}(\MSL,A)$.
The obstruction $\varphi \circ \mu^{SL} - \mu_{A} \circ (\varphi \wedge \varphi)$
to having a morphism of monoids lies in a similarly defined subgroup
of $Hom_{SH(S)}(\MSL \wedge \MSL, A)$ (Theorem \ref{T:univ.SL}).


It would be interesting to know if these obstruction subgroups vanish for Witt groups
and hermitian $K$-theory.
The necessary calculations are likely very close to Balmer and Calm\`es's computation of Witt groups
of Grassmannians \cite{Balmer:2008fk}.

Our $\MSp$ is defined similarly with an action of $Sp_{2n}$ on the spaces
$BSp_{2n}$ and $\MSp_{2n}$.  The actions of the subgroups $\sym_{n} \subset Sp_{2n}$ make
$\MSp$ a commutative monoid in the category of symmetric $T^{\wedge 2}$-spectra.
For $\MSp$ we can do much more than for $\MSL$ because
we have the quaternionic projective bundle theorem \cite[Theorem 8.2]{Panin:2010fk}
for symplectically oriented cohomology theories.  Therefore for any symplectically oriented
cohomology theory $A^{*,*}$ we have Borel classes for symplectic bundles, and we
can compute the cohomology of quaternionic Grassmannians
\cite[\S 11]{Panin:2010fk}
and of the spaces $BSp_{2r}$ and $\MSp_{2r}$
(\S\S\ref{S:projective.bundle}--\ref{S:cohom.BSp.MSp}).
Our main result is the following theorem.

%
%
%

\begin{thm}
\label{T:main.Sp}
Let $(A,\mu,e)$ be a commutative monoid in $SH(S)$.  Then the following sets are in canonical
bijection\textup{:}

\parens{a} symplectic Thom structures on the bigraded $\epsilon$-commutative ring cohomology theory
$(A^{*,*},\partial,\times,1_{A})$ such that for the trivial rank $2$
bundle $\Aff^{2} \to \pt$ we have
$\thom(\Aff^{2},\omega_{2}) = \Sigma_{T}^{2}1_{A}$ in $A^{4,2}(T^{\wedge 2})$,

\parens{b} Borel structures on
$(A^{*,*},\partial,\times,1_{A})$ for which
$b_{1}(\shf U_{HP^{1}},\phi_{HP^{1}}) \in A^{4,2}(HP^{1},h_{\infty}) \subset A^{4,2}(HP^{1})$
corresponds to $-\Sigma_{T}^{2}1_{A}$ in $A^{4,2}(T^{\wedge 2})$ under the canonical
motivic homotopy equivalence $(HP^{1},h_{\infty}) \simeq T^{\wedge 2}$,

\parens{c} Borel classes theories on $(A^{*,*},\partial,\times, 1_{A})$
with the same normalization condition on $b_{1}(\shf U_{HP^{1}},\phi_{HP^{1}})$ as in \parens{b},

\parens{d} symplectic Thom classes theories on
$(A^{*,*},\partial,\times,1_{A})$ such that for the trivial rank $2$
bundle $\Aff^{2} \to \pt$ we have
$\thom(\Aff^{2},\omega_{2}) = \Sigma_{T}^{2}1_{A}$ in $A^{4,2}(T^{\wedge 2})$,

\parens{$\alpha$} classes $\vartheta \in A^{4,2}(\MSp_{2})$ with
$\vartheta |_{T^{\wedge 2}} = \Sigma_{T}^{2}1_{A}$ in $A^{4,2}(T^{\wedge 2})$,

\parens{$\beta$} classes $\varrho \in A^{4,2}(HP^{\infty},h_{\infty})$ with
$\varrho|_{HP^{1}} \in A^{4,2}(HP^{1},h_{\infty})$ corresponding to
$-\Sigma_{T}^{2}1_{A} \in A^{4,2}(T^{\wedge 2})$ under the canonical motivic homotopy equivalence
$(HP^{1},h_{\infty}) \cong T^{\wedge 2}$,

\parens{$\delta$} sequences of classes
$\boldsymbol{\vartheta} = (\vartheta_{1},\vartheta_{2},\vartheta_{3},\dots)$
with $\vartheta_{r} \in A^{4r,2r}(\MSp_{2r})$ for each $r$ satisfying
$\mu_{rs}^{*} \vartheta_{r+s} = \vartheta_{r} \times \vartheta_{s}$ for all $r,s$,
and $\vartheta_{1} |_{T^{\wedge 2}} = \Sigma_{T}^{2}1_{A}$,

\parens{$\varepsilon$} morphisms $\varphi \colon (\MSp,\mu^{Sp},e^{Sp}) \to (A,\mu,e)$
of commutative monoids
in $SH(S)$.
\end{thm}

The bijections are explicit and are given in a series of theorems in the last part of the paper.
The presence of $(\varepsilon)$ among them is the universality of $\MSp$ as a symplectically
oriented theory.

The equivalence of (a), (b), (c) and (d) was already shown in \cite{Panin:2010fk} in a
different axiomatic context.
The ability of the motivic language used here to handle tautological classes such as the
$(\alpha)$, $(\beta)$ and $(\delta)$ is very useful by itself.  But
our main new observation is that in the motivic unstable homotopy category $H_{\bullet}(S)$
we have a commutative diagram (Theorem \ref{T:MSp=BSp/BSp}).
\[
\xymatrix @M=5pt {
& BSp_{2r} \ar@/_/[ld]_-{\text{\textup{structure map}}} \ar@/^/[rd]^-{\text{\textup{quotient}}}
\\
\MSp_{2r} \ar@{<->}[rr]^-{\cong}_-{\text{\textup{$\Aff^{N}$-bundles and excision}}}
&& BSp_{2r}/BSp_{2r-2}.
}
\]
What is surprising about this diagram is that it is the homotopy colimit of diagrams \eqref{E:amazing.1}
of finite-dimensional schemes and their quotient spaces which have a fourth side which is an inclusion
of quaternionic Grassmannians of different dimensions
which is in no way a
motivic equivalence.
But in the infinite-dimensional colimit the fourth side becomes $\Aff^{1}$-homotopic to the identity map
of $BSp_{2r}$, and the picture simplifies significantly.

The fact that this diagram is three-sided instead of four-sided
helps us to see more conceptual proofs of two of the trickier points of
\cite{Panin:2010fk}.  One is the construction of the higher-rank symplectic Thom classes and the
proof of their multiplicativity.  The commutativity of the diagram and the computations of the
cohomology of quaternionic Grassmannians imply that given a symplectic Thom structure on $A^{*,*}$,
the pullback along the structure map gives an injection $A^{*,*}(\MSp_{2r}) \to A^{*,*}(BSp_{2r})$,
and the isomorphism $A^{*,*}(\pt)[[b_{1},\dots,b_{r}]]^{\homog} \cong A^{*,*}(BSp_{2r})$
defined by the symplectic Thom structure identifies the image of $A^{*,*}(\MSp_{2r})$
with the principal two-sided ideal generated by $b_{r}$.  This makes it easy to define
the higher-rank tautological symplectic Thom classes (the $\vartheta_{r}$ of ($\delta$)) with the classes
of $A^{*,*}(\MSp_{2r})$ identified with $(-1)^{r}b_{r} \in A^{*,*}(\pt)[[b_{1},\dots,b_{r}]]^{\homog}$.
Their multiplicativity is also easily established.

The other tricky point of \cite{Panin:2010fk} for which the diagram helps is
the reconstitution of the symplectic Thom structure
from the Borel structure.  The tautological rank $2$ Thom class is a $\vartheta \in A^{4,2}(\MSp_{2})$,
and it is tempting to identify it (up to sign) with the tautological rank $2$ Borel class
$\varrho \in A^{4,2}(BSp_{2}/BSp_{0}) = A^{4,2}(HP^{\infty},h_{\infty})$ using the horizontal
motivic homotopy equivalence.  But the Borel is actually (up to sign) the pullback of $\vartheta$
along the structure map of the Thom space.  For the three-sided diagram this is no problem:
in $H_{\bullet}(S)$ the structure map $HP^{\infty} \to \MSp_{2}$ is the composition of the
horizontal isomorphism $(HP^{\infty},h_{\infty}) \cong \MSp_{2}$
with the pointing map.

%
%
%
%
%
%
%
%

It is not difficult to define spectra $\mathbf{MO}$ and $\mathbf{MSO}$ which resemble formally
$\MGL$ and our $\MSL$ and $\MSp$.  However, our proof of even our
most basic result about $\MSL$ (Theorem \ref{T:thom.phi})
uses the fact that special linear bundles are locally trivial in the
Zariski topology.  So we omit
$\mathbf{MO}$ and $\mathbf{MSO}$.

\section{Preliminaries}
\label{S:prelim}

Let $S$ be a noetherian scheme of finite Krull dimension, and let $\Sm/S$ be the category of smooth
quasi-projective schemes over $S$.
We will assume that $S$ admits an ample family of line bundles
so that for any $X$ in $\Sm/S$ there exists an affine bundle $Y \to X$ with $Y$ an
affine scheme.  This condition is used a number of times in this paper, and it was also used in the
proof of the symplectic splitting principle in \cite[Theorem 10.2]{Panin:2010fk}.

The category $\SmOp/S$ has objects $(X,U)$ where $X$ is in $\Sm/S$
and $U \subset X$ is an open subscheme.  A morphism $f \colon (X,U) \to (X',U')$ in $\SmOp/S$
is a morphism
$f \colon X \to X'$ of $S$-schemes with $f(U) \subset U'$.
We often write $X$ in place of $(X,\varnothing)$.

A \emph{bigraded ring cohomology theory} $(A^{*,*},\partial, \times, 1)$ on $\SmOp/S$ is
a contravariant functor $A^{*,*}$  from $\SmOp/S$ to the category of bigraded abelian groups
which satisfies \'etale excision and $\Aff^{1}$-homotopy invariance and which has
localization long exact sequences
\[
\cdots \to A^{*,*}(X,U) \to A^{*,*}(X) \to A^{*,*}(U) \xra{\partial} A^{*+1,*}(X,U) \to \cdots.
\]
The $\times$ product is assumed to be functorial, bilinear, associative, and compatible with the
bigrading with a two-sided unit $1$.

In this paper we work mainly with the motivic unstable and stable homotopy categories $H_{\bullet}(S)$
and $SH(S)$.  The former is the homotopy category of a model category $\M_{\bullet}(S)$
of pointed motivic spaces over $S$ with motivic weak equivalences.  There are several versions
of this model category with different underlying categories and with different choices of
fibrations and cofibrations.
See \cite{Dundas:2003aa,Isaksen:2005aa,Jardine:2000aa,Morel:1999ab,Panin:2009aa,Voevodsky:1998kx,Voevodsky:2007aa}
among other papers.
It is not essential which of these model category structures is used.
However, we give geometric constructions symmetric $T$- and $T^{\wedge 2}$-spectra using the
the Morel-Voevodsky object $T = \Aff^{1}/(\Aff^{1}-0)$ itself.  So the best adapted model
structures is perhaps the flasque motivic model category of \cite{Isaksen:2005aa}
which is known to be cellular
(so we can apply Hovey's results \cite{Hovey:2001aa} in the proof of Theorem \ref {T:comparison})
but for which $T$ is cofibrant.

The category $\M_\bullet(S)$ is equipped with a symmetric monoidal structure
$(\M_\bullet(S), \wedge, S^{0})$. The smash-product is taken sectionwise, and
$S^{0}$
is the constant simplicial zero-sphere. This smash-product induces a smash-product
on $H_{\bullet}(S)$ such that the natural functor
$\M_\bullet(S) \to H_{\bullet}(S)$
becomes a strict symmetric monoidal functor
$(\M_\bullet(S), \wedge, S^{0}) \to (H_\bullet(S), \wedge, S^{0})$.

We set $T=\Aff^1/(\Aff^1-0)$.
A \emph{$T$-spectrum} $E$ is a sequence
$(E_0,E_1,\dotsc)$ of pointed motivic spaces equipped with a sequence of structure maps
$\sigma_{n}\colon E_n \wedge T \to E_{n+1}$ of pointed motivic spaces.
A morphism of $T$-spectra $E \to E^{\prime}$ is a sequence of maps
$f_n\colon E_n \to E'_n$ of pointed motivic spaces which commute
with the structure maps. The category of $T$-spectra $Sp(\M_\bullet(S),T)$ can be equipped
with a motivic stable model structure as in
\cite{Jardine:2000aa,Voevodsky:1998kx, Voevodsky:2007aa}. Its homotopy category
is  $SH(S)$.
%
%
This category can be equipped with a structure of a symmetric monoidal category
$(SH(S), \wedge, {\boldsymbol 1})$, satisfying the conclusions of
\cite[Theorem 5.6]{Voevodsky:1998kx}.
The unit ${\boldsymbol 1}$ of that
monoidal structure is the $T$-sphere spectrum
$\sphere =(S^{0},T,T\wedge T, \dots )$.


Every $T$-spectrum $E = (E_0,E_1,\dotsc)$ represents a cohomology theory on the category
of pointed motivic spaces $\M_\bullet(S)$.
Namely, let $S^n_s$ and $S^n_t$ be as in
\cite[(16)]{Voevodsky:1998kx}.
Let
$S^{p,q}=S^{p-q}_s \wedge S^q_t$.
We write
$$E^{p,q}(A)=Hom_{SH(S)}(\Sigma^{\infty}_{T} A, E \wedge S^{p,q})$$
as in \cite[\S 6]{Voevodsky:1998kx}.
There is a canonical element in $E^{2n,n}(E_n)$, denoted as
$$\Sigma^\infty_{T}E_n(-n) \xra{u_n} E.$$
It is represented by
the canonical map $(\ast,\dotsc,\ast,E_n,E_n\wedge T,\dotsc) \to (E_0,E_1,\dotsc,E_n,\dotsc)$
of $T$-spectra.


A \emph{$T$-ring spectrum} is a monoid
$(E,\mu,e)$
in
$({SH(S)},\wedge, {\bf 1})$.
The cohomology theory $E^{\ast,\ast}$ defined by a
$T$-ring spectrum is a ring cohomology theory on $\M_{\bullet}(S)$. To see recall
the standard isomorphism
$S^{i,j} \wedge S^{k,l} \cong S^{i+k,j+l}$
given by the composition
$$(S^{i-j}_s \wedge S^j_t) \wedge (S^{k-l}_s \wedge S^l_t) \cong
(S^{i-j}_s \wedge S^{k-l}_s) \wedge (S^{j}_t \wedge S^{l}_t)
\cong S^{i-j+k-l}_s \wedge S^{j+l}_t.$$
For $X, Y \in \M_\bullet(k)$
let
$\alpha \colon \Sigma^{\infty}(X) \to S^{i,j} \wedge E$
and
$\beta \colon \Sigma^{\infty}(Y) \to S^{k,l} \wedge E$
be elements of
$E^{i,j}(X)$ and $E^{k,l}(Y)$ respectively.
Following Voevodsky
\cite{Voevodsky:1998kx} define
$\alpha \times \beta \in E^{i+k,j+l}(X \wedge Y)$
as the composition
\[\Sigma^{\infty}(X \wedge Y) \cong \Sigma^{\infty}(X) \wedge \Sigma^{\infty}(Y) \xra{\alpha \wedge \beta}
E \wedge S^{i,j} \wedge E \wedge S^{k,l} \cong E \wedge E \wedge S^{i+k,j+l} \xra{id \wedge \mu}
E \wedge S^{i+k,j+l}.
\]
This gives a functorial product which is associative, has a two-sided unit, and takes cofibration sequences to
long exact sequences.

A \emph{commutative} $T$-ring
spectrum is a commutative monoid
$(E,\mu,e)$
in
$(SH(S),\wedge, {\bf 1})$.  To describe the properties
of the associated cohomology theory we make some definitions.

\begin{defn}
\label{epsilon}
Let $in_T \colon T \to T$ be a morphism of pointed motivic spaces induced by the morphism
$\Aff^1 \to \Aff^1$ sending $t \mapsto -t$. One has the equality
$$Hom_{SH(S)}(\pt_{+}, \pt_{+})=Hom_{SH(S)}(T,T).$$
We write
$in$ for $in_T$
regarded as an element of
$Hom_{SH(S)}(\pt_{+}, \pt_{+})$.
For a commutative monoid $(A,m,e)$
set $\epsilon=in^*(e) \in A^{0,0}(\pt_{+})$.
\end{defn}

\begin{rem}
\label{epsilontimesepsilon}
The morphism $\Aff^2 \to \Aff^2$ which sends
$(t_1,t_2) \mapsto (-t_1,-t_2)$ is $\Aff^1$-homotopic to the identity morphism
because
$\bigl(
\begin{smallmatrix}
-1 & 0 \\ 0 & -1
\end{smallmatrix}
\bigr)
\in SL_{2}(\mathbb Z)$
is a product of elementary matrices.
Whence we have
$\epsilon \times \epsilon= e \in A^{0,0}(\pt_{+})$.
\end{rem}


\begin{defn}
A ring cohomology theory on $\M_{\bullet}(S)$ is
\emph{$\epsilon$-commutative} if for
$\alpha \in E^{i,j}(X)$, $\beta \in E^{k,l}(Y)$
one has
$\sigma^*_{X,Y}(\alpha \times \beta)=(-1)^{ik}\epsilon^{jl}(\beta \times \alpha) \in
E^{i+k,j+l}(X \wedge Y)$
where $\sigma_{X,Y} \colon X \wedge Y \to Y \wedge X$ switches the factors.
\end{defn}

Thus $\epsilon$-commutativity is a specific form of bigraded commutativity.

\begin{thm}[Morel]
\label{BO**asringcohomology}
Let $(E,\mu,e)$ be a commutative monoid in $SH(S)$.
Then the data $(E^{*,*}, \partial, \times, e)$ is an associative and
$\epsilon$-commutative ring cohomology theory on
$\M_\bullet(k)$\textup{;}
$e \in E^{0,0}(S^{0})$ is the two-sided unit of the ring structure.
\end{thm}

Note, that if $(i,j)=(2m,2n)$ or $(k,l)=(2m,2n)$, then
$\sigma^*_{A,B}(\alpha \times \beta)=\beta \times \alpha$.

\section{Commutative
{$T$- and $T^{\wedge 2}$}{T- and T\^{ }2}-monoids}

We compare the categories of symmetric $T$-spectra and symmetric $T^{\wedge 2}$-spectra.
Recall the definition for $K = T$ or $K = T^{\wedge 2}$.

\begin{defn}
A \emph{symmetric $K$-spectrum} $E$ is a sequence of pointed spaces $(E_{0},E_{1},E_{2},\dots)$
with each $E_{n}$ equipped with an action of the symmetric group $\sym_{n} \times E_{n} \to E_{n}$
and with a morphism $\sigma_{n} \colon E_{n} \wedge K \to E_{n+1}$ such that the induced maps
$E_{n} \wedge K^{\wedge m} \to E_{n+m}$ are $(\sym_{n} \times \sym_{m})$-equivariant
for all $n$ and $m$.
\end{defn}

The categories of symmetric $T$- and $T^{\wedge 2}$-spectra both have a symmetric monoidal
product $\wedge$.  They are symmetric monoidal model categories for the stable model structure
\cite{Hovey:2001aa,Jardine:2000aa}.

\begin{thm}
\label{T:comparison}
The homotopy categories of $Sp^{\sym}(\M_{\bullet}(S),T)$ and of
$Sp^{\sym}(\M_{\bullet}(S),T^{\wedge 2})$ are equivalent symmetric monoidal categories.
\end{thm}

\begin{proof}
The proof of this theorem is essentially the same as that given for topological $S^{1}$- and $S^{2}$-spectra in
\cite[Theorem A.44]{Panin:2009aa}.  The inclusion
$Sp^{\sym}(\M_{\bullet}(S),T) \to Sp^{\sym}(Sp^{\sym}(\M_{\bullet}(S),T),T^{\wedge 2})$
is a Quillen equivalence by \cite[Theorem 9.1]{Hovey:2001aa} because ${-}\wedge T^{\wedge 2}$
is a Quillen self-equivalence of $Sp^{\sym}(\M_{\bullet}(S),T)$.  Similarly the inclusion
$Sp^{\sym}(\M_{\bullet}(S),T^{\wedge 2}) \to Sp^{\sym}(Sp^{\sym}(\M_{\bullet}(S),T^{\wedge 2}),T)$
is a Quillen equivalence.  The two categories of symmetric
bispectra are isomorphic with the identical stable model
structure by arguments like those used in the proof of \cite[Theorem 10.1]{Hovey:2001aa}.
Hovey's work requires that the model structure have certain properties, but in the flasque model structure
\cite{Isaksen:2005aa} these properties hold, and $T$ and $T^{\wedge 2}$ are cofibrant.

The symmetric monoidal structures are the same
because (i) the inclusions of the categories of symmetric spectra in the categories of symmetric
bispectra are symmetric monoidal
functors like any inclusion $\Sigma_{K}^{\infty} \colon \mathcal C \to Sp^{\sym}(\mathcal C,K)$, and
(ii) the symmetric monoidal structures on the two isomorphic  categories of symmetric  bispectra
are the same.
\end{proof}

For natural numbers $m,n$ we denote by $c_{m,n} \in \sym_{m+n}$
the \emph{$(m,n)$-shuffle permutation}.
It acts by $c_{m,n}(i) = i+n$ for $1 \leq i \leq m$ and $c_{m,n}(i) = i-m$ for $m+1 \leq i \leq m+n$.

\begin{defn}
\label{D:T.monoid}
A \emph{commutative $K$-monoid} $E$ in $\M_{\bullet}(S)$ is a sequence of
pointed motivic spaces
$(E_{0},E_{1},E_{2},\dots)$ with each space equipped with an action $\sym_{n} \times E_{n} \to E_{n}$
of the symmetric group, plus morphisms
\begin{equation}
\label{E:T.monoid.structure}
\begin{aligned}
e_{0} \colon & \boldsymbol{1}_{\M_{\bullet}(S)} \to E_{0},
\\
e_{1} \colon & K \to E_{1}
\\
\mu_{mn} \colon & E_{m} \wedge E_{n} \to E_{m+n}
\end{aligned}
\end{equation}
in $\M_{\bullet}(S)$ such that each $\mu_{mn}$ is $(\sym_{m} \times \sym_{n})$-equivariant and such that the compositions
\begin{equation}
\label{E:T.unit}
\begin{gathered}
E_{n} \xra{\cong} E_{n} \wedge \boldsymbol{1}_{M_{\bullet}(S)} \xra{1 \wedge e_{0}} E_{n}\wedge E_{0}
\xra{\mu_{n0}} E_{n}
\\
E_{n} \xra{\cong} \boldsymbol{1}_{M_{\bullet}(S)} \wedge E_{n}  \xra{e_{0} \wedge 1} E_{0}\wedge E_{n}
\xra{\mu_{0n}} E_{n}
\end{gathered}
\end{equation}
are the identity maps, and the diagrams
\begin{equation}
\vcenter{
\xymatrix @M=5pt @C=30pt {
E_{\ell } \wedge E_{m} \wedge E_{n} \ar[r]^-{\mu_{\ell m}\wedge 1} \ar[d]_-{1 \wedge \mu_{mn}}
&
E_{\ell +m} \wedge E_{n} \ar[d]^-{\mu_{\ell +m,n}}
\\
E_{\ell } \wedge E_{m+n} \ar[r]^-{\mu_{\ell ,m+n}}
&
E_{\ell +m+n}
}}
\qquad \quad
\label{E:commutativity}
\vcenter{
\xymatrix @M=5pt @C=35pt {
E_{m} \wedge E_{n}
\ar[d]_-{\text{switch}}
\ar[r]^-{\mu_{mn}}
&
E_{m+n}
\ar[d]_-{\cong}^-{c_{m,n}}
\\
E_{n} \wedge E_{m}
\ar[r]^-{\mu_{nm}}
&
E_{n+m}
}}
\end{equation}
commute in $\M_{\bullet}(S)$ with $c_{m,n}$ the isomorphism given by the action
of the $(m,n)$-shuffle permutation.
\end{defn}

\begin{thm}
Let $E$ be a commutative $K$-monoid.  Define maps $\sigma_{n}$ as the compositions
\begin{equation}
\label{E:T.bonding}
\sigma_{n} \colon E_{n} \wedge K \xra{1 \wedge e_{1}} E_{n} \wedge E_{1} \xra{\mu_{n1}} E_{n+1}.
\end{equation}
Then the spaces $(E_{0},E_{1},E_{2},\dots)$ equipped with the actions $\sym_{n} \times E_{n} \to E_{n}$
and the bonding maps $\sigma_{n}$ form a symmetric $K$-spectrum $E$.
Moreover, the morphisms
$\mu \colon E \wedge E \to E$ induced by the $\mu_{mn}$
and $e \colon \Sigma_{K}^{\infty} \boldsymbol{1}_{\M_{\bullet}(S)}\to E$
composed of the maps $e_{n} \colon K^{\wedge n} \to E_{n}$ induced by $e_{0}$, $e_{1}$ and
the $\mu_{mn}$ make $(E,\mu,e)$
a commutative monoid in $Sp^{\sym}(\M_{\bullet}(S),K)$.
\end{thm}

\begin{proof}[Sketch of proof]
To show that $E$ is a symmetric $K$-spectrum one has to verify that the induced maps
$E_{n} \wedge K^{\wedge j} \to E_{n+j}$ are $(\sym_{n} \times \sym_{j})$-equivariant.
To show that the maps $\mu_{mn}$ define a morphism $E \wedge E \to E$
one has to verify that they are $K$-linear and $K$-bilinear in the sense of
\cite[(4.6)--(4.7)]{Jardine:2000aa}.  One has to verify that each $e_{n}$ is $\sym_{n}$-equivariant.
Finally, one has to verify the commutative monoid axioms.
All the verifications are formal, straightforward and left to the reader.
\end{proof}

\section{The symmetric
{$T^{\wedge 2}$}{T\^{ }2}-spectrum
{$\MSL$}{MSL}}

We construct a commutative $T$-monoid $\MSL$.
Each space $\MSL_{n}$ comes equipped
with an action of $GL_{n}$ such that the multiplication maps $\mu_{mn}$ of the monoid
structure are $(GL_{m} \times GL_{n})$-equivariant.  We then get actions of the $\sym_{n}$
from the embeddings $\sym_{n} \to GL_{n}$ given by permutation matrices.
The need for an action of $GL_{n}$ with fixed points --- necessary for the proper definition of the unit maps
---
rather than merely of $SL_{n}$ is the delicate part of the construction.

We begin by reviewing the construction of $\MGL$ originally done in
\cite[\S 6.3]{Voevodsky:1998kx} and of its monoid structure given in
\cite{Panin:2008db,Vezzosi:2001aa}.

For each integer $n \geq 0$ let $\Gamma_{n} = \OO_{S}^{\oplus n}$ be the trivial rank-$n$ vector bundle.
For each integer $p \geq 1$ let $Gr(n,np) =  Gr(n,\Gamma_{n}^{\oplus p})$.  Let
$\TGL{n,np} \to Gr(n,np)$ be the tautological
subbundle.   The inclusions
$(1,0) \colon \Gamma_{n}^{\oplus p} \to \Gamma_{n}^{\oplus p} \oplus \Gamma_{n}
= \Gamma_{n}^{\oplus p+1}$
induce closed embeddings $Gr(n,np) \hra Gr(n,np+n)$ and monomorphisms
$\Th \TGL{n,np} \to \Th \TGL{n,np+n}$
of Thom spaces.  We set
\begin{gather*}
BGL_{n} = \colim\nolimits_{p\in \NN} Gr(n,np),
\\
\TGL{n,n\infty} = \colim\nolimits_{p\in \NN} \TGL{n,np},
\\
\MGL_{n} = \colim\nolimits_{p\in \NN} \Th \TGL{n,np}.
\end{gather*}
The diagonal action of $GL_{n} = GL(\Gamma_{n})$ on each $Gr(n,np) = Gr(n,\Gamma_{n}^{\oplus p})$
is compatible with the inclusions over increasing $p$.  Moreover, the $\TGL{n,np}$ are $GL_{n}$-equivariant vector bundles.
This induces actions
\[
\sym_{n} \times \MGL_{n} \subset GL_{n} \times \MGL_{n} \to \MGL_{n}.
\]
Concatenation of bases induces isomorphisms $\Gamma_{m} \oplus \Gamma_{n} = \Gamma_{m+n}$
which induce $(GL_{m} \times GL_{n})$-equivariant maps
\begin{equation*}
\oplus \colon
Gr(m,mp) \times Gr(n,np) \to Gr(m+n,mp+np)
\end{equation*}
and therefore $(GL_{m} \times GL_{n})$-equivariant maps
\begin{gather*}
\oplus \colon BGL_{m} \times BGL_{n} \to BGL_{m+n},
\\
\mu_{mn}^{GL} \colon \MGL_{m} \wedge \MGL_{n} \to \MGL_{m+n}.
\end{gather*}
Finally each $Gr(n,np)$ is pointed by the point corresponding to the trivial rank-$n$ subbundle
\[
\Gamma_{n} \xra{(1,0,\dots,0)} \Gamma_{n}^{\oplus p}.
\]
In the colimit this gives a $S$-valued point $x_{n}$ of $BGL_{n}$, which is fixed by the action of $GL_{n}$.
The Thom space of the fiber
of $\mathcal T(n,n\infty)$ over $x_{n}$ is $\Gamma_{n} \cong \Aff^{n}$, and the inclusion
$x_{n} \hra BGL_{n}$ induces a map of Thom spaces
\[
e_{n}^{GL} \colon T^{\wedge n} \to \MGL_{n}.
\]

\begin{defn}
The \emph{algebraic cobordism spectrum} $\MGL$ is the commutative monoid in the category of
symmetric $T$-spectra associated to the
commutative $T$-monoid composed of
the spaces $\MGL_{n}$, the actions $\sym_{n} \times \MGL_{n} \to \MGL_{n}$, the maps
$e_{0}^{GL} \colon \pt_{+} \to \MGL_{0}$ and $e_{1}^{GL} \colon T \to \MGL_{1}$
and the maps $\mu_{mn}^{GL} \colon \MGL_{m} \wedge \MGL_{n} \to \MGL_{m+n}$.
\end{defn}

We now move on to defining $\MSL$. We begin with the spaces.
For $n = 0$ we have $SL_{0} = GL_{0} = \{1\}$.  So we set $BSL_{0} = \pt$.  The Thom space of
a zero vector bundle over a scheme $X$ is the externally pointed space $X_{+}$.
So we set $\MSL_{0} = \pt_{+}$.

Now suppose $n > 0$.
Over each $Gr(n,np)$ there is the line bundle $\OO_{Gr(n,np)}(-1) = \det \TGL{n,np}$.
Removing the zero section gives a smooth scheme
\[
\GrSL(n,np) = \OO_{Gr(n,np)}(-1) - Gr(n,np).
\]
The projection
\[
\pi = \pi_{n,np} \colon \GrSL(n,np)  \to Gr(n,np)
\]
is a principal $\GG_{m}$-bundle.
Write
\[
\TSL{n,np}  = \pi^{*} \TGL{n,np}.
\]
The inclusion $\GrSL(n,np) \hra \OO_{Gr(n,np)}(-1)$ and the cartesian diagram
\[
\xymatrix @M=5pt {
\pi^{*}\OO_{Gr(n,np)}(-1) \ar[r] \ar[d] \ar@{}[rd]|{\Box}
& \OO_{Gr(n,np)}(-1) \ar[d]
\\
\GrSL(n,np) \ar[r]_-{\pi}
& Gr(n,np)
}
\]
gives a nowhere vanishing section of $\pi^{*}\OO_{Gr(n,np)}(-1) = \det \TSL{n,np}$.
The corresponding isomorphism
$\lambda_{n,np} \colon \OO_{\GrSL(n,np)} \cong \det \TSL{n,np}$
makes $(\TSL{n,np}, \lambda_{n,np})$ the \emph{tautological special linear bundle}
over $\GrSL(n,np)$.

We set
\begin{gather*}
BSL_{n} = \colim\nolimits_{p\in \NN} \GrSL(n,np),
\\
\TSL{n,n\infty} = \colim\nolimits_{p\in \NN} \TSL{n,np},
\\
\MSL_{n} = \colim\nolimits_{p\in \NN} \Th \TSL{n,np}.
\end{gather*}

We next define the multiplication maps.  Morphisms of $S$-schemes $X \to \GrSL(n,np)$ are in bijection
with pairs $(f,\lambda)$ with $f \colon X \to Gr(n,np)$ a morphism and
$\lambda \colon \OO_{X} \cong \det f^{*} \TGL{n,np}$ an isomorphism.
There are unique maps
\[
(\oplus, \otimes) \colon \GrSL(m,mp) \times \GrSL(n,np) \to \GrSL(m+n,mp+np)
\]
corresponding to the morphisms of representable functors
\[
\begin{array}{ccc}
Hom(X,\GrSL(m,mp)) \times Hom(X,\GrSL(n,np)) & \lra & \Hom(X,\GrSL(m+n,mp+np))
\\
((f,\lambda),(g,\lambda_{1})) & \longmapsto & (f \oplus g, \lambda \otimes \lambda_{1}).
\end{array}
\]
They induce maps
\begin{gather*}
(\oplus,\otimes) \colon BSL_{m} \times BSL_{n} \to BSL_{m+n}
\\
\mu^{SL}_{mn} \colon \MSL_{m} \wedge \MSL_{n} \to \MSL_{m+n}.
\end{gather*}

We now discuss the group actions.
%
Since $\TGL{n,np}$ is a $GL_{n}$-equivariant bundle over $Gr(n,np)$,
there is an induced action of $GL_{n}$ on the complement of the zero section of the
determinant line bundle.
This is an action $GL_{n} \times \GrSL(n,np) \to \GrSL(n,np)$.  In the colimit this gives an action
\(
GL_{n} \times BSL_{n} \to BSL_{n}.
\)
But there is a problem.

The unit maps $e_{n}^{GL} \colon T^{\wedge n} \to \MGL_{n}$ were defined using points
$x_{n} \colon \pt \to BGL_{n}$ which were fixed under the action of $GL_{n}$.
To define unit maps for a $T$-monoid $\MSL$ we need fixed points for the action of at least
$\sym_{n}$ on $BSL_{n}$,
preferably lying over $x_{n}$.
We have a cartesian diagram
\[
\xymatrix @M=5pt {
\GG_{m} \ar[r] \ar[d] \ar@{}[rd]|{\Box}
&
BSL_{n} \ar[d]
\\
\pt \ar[r]_-{x_{n}}
&
BGL_{n}.
}
\]
The action of $GL_{n} = GL(\Gamma_{n})$ on the fiber $\Gamma_{n}$ of $\TGL{n,n\infty}$
over the fixed point $x_{n}$ is the standard representation of $GL_{n}$.  So the induced action on the fiber
$\GG_{m}$ over $x_{n}$ is $g \cdot t = \det(g)t$.  Thus there are fixed points
for the action of the alternating group $\mathfrak A_{n} \subset SL_{n}$ on $BSL_{n}$
lying over the fixed point $x_{n} \in BGL_{n}(\pt)$ used to define the unit maps on $\MGL_{n}$
but not for the action of $\sym_{n} \subset GL_{n}$ (except in characteristic $2$).

So we use the embedding $\sym_{n} \subset Sp_{2n} \subset SL_{2n}$ which sends $\sigma \in \sym_{n}$
to the permutation matrix associated to $\bar\sigma \in \sym_{2n}$ where we have
$\bar\sigma(2i-1) = 2\sigma(i)-1$ and $\bar\sigma(2i) = 2\sigma(i)$.  This gives us an action
$\sym_{n} \times BSL_{2n} \to BSL_{2n}$ which fixes pointwise the fiber over $x_{n}$.

Therefore we define the spaces of the commutative $T^{\wedge 2}$-monoid $\MSL$ to be the
$\MSL_{2n}$.   Each is equipped with the action of $\sym_{n} \times \MSL_{2n} \to \MSL_{2n}$
induced by the action of $SL_{2n}$.

%
%

We now define the unit maps.
Points $\pt \to BSL_{n}$ lifting the point $x_{n} \colon \pt \to BGL_{n}$ are in bijection
with isomorphisms
$\lambda \colon \OO_{S} \cong \det x_{n}^{*} \TGL{n,n\infty} = \Lambda^{n} \Gamma_{n}$.
Let $f_{1},\dots,f_{n}$ be the standard basis of $\Gamma_{n} = \OO_{S}^{\oplus n}$.
We let $y_{n} \colon \pt \to BSL_{n}$ be the lifting of $x_{n}$ corresponding to
$\lambda = f_{1} \wedge \cdots \wedge f_{n}$.
The fiber of $\TSL{n,n\infty}$ over $y_{n}$
is $\Gamma_{n} \cong \Aff^{n}$, and we let
\[
e_{n}^{SL} \colon T^{\wedge n} \to \MSL_{n}
\]
be the map of Thom spaces induced by $y_{n}$.  It is $SL_{n}$-equivariant.
Note that $e_{0}^{SL} \colon \pt_{+} \to \MSL_{0} = \pt_{+}$ is the identity.

Having identified the components of the structure $\MSL$, we have to assemble them.
It appears as if $\MSL$ is a
commutative monoid in the category of alternating $T$-spectra.
We do not know how to work in that category.  But there is underlying structure.

\begin{defn}
The \emph{algebraic special linear cobordism spectrum} $\MSL$ refers to three
related objects.

(a)
The commutative monoid in the category of symmetric $T^{\wedge 2}$-spectra associated
to the commutative $T^{\wedge 2}$-monoid composed of
the spaces $\MSL_{2n}$, the actions $\sym_{n} \times \MSL_{2n} \to \MSL_{2n}$, the maps
$e_{0}^{SL} \colon \pt_{+} \to \MSL_{0}$ and $e_{2}^{SL} \colon T^{\wedge 2} \to \MSL_{2}$
and the maps $\mu_{2m,2n}^{SL} \colon \MSL_{2m} \wedge \MSL_{2n} \to \MSL_{2m+2n}$.

(b) The $T$-spectrum with spaces $\MSL_{n}$, bonding maps
$\MSL_{n} \wedge T \to \MSL_{n} \wedge \MSL_{1} \to \MSL_{n+1}$ induced by $e_{1}^{SL}$
and $\mu_{n,1}^{SL}$, equipped with the morphism of $T$-spectra
$e \colon \Sigma_{T}^{\infty}\pt_{+} \to \MSL$
and the structural maps $\mu_{mn}^{SL}$.

(c) Their common underlying $T^{\wedge 2}$-spectrum.
\end{defn}

The properties of the commutative monoid structure that we require are given in the following theorem.

\begin{thm}
\label{T:MSL.structure}
The $(\MSL,\mu^{SL},e^{SL})$ is a commutative monoid in $SH(S)$, and the canonical maps
$u_{n} \colon \Sigma_{T}^{\infty} \MSL_{n}(-n) \to \MSL$ and the $\mu_{mn}^{SL}$ make
the following diagram commute for all $m$ and $n$
\begin{equation}
\label{E:MSLm.wedge.MSLn}
\vcenter{
\xymatrix @M=5pt @C=50pt {
\Sigma_{T}^{\infty}\MSL_{m}(-m) \wedge \Sigma_{T}^{\infty}\MSL_{n}(-n)
\ar[r]^-{\Sigma_{T}^{\infty}\mu_{mn}^{SL}} \ar[d]_-{u_{m} \wedge u_{n}}
& \Sigma_{T}^{\infty}\MSL_{m+n}(-m-n) \ar[d]^-{u_{m+n}}
\\
\MSL \wedge \MSL \ar[r]^-{\mu^{SL}}
&
\MSL.
}}
\end{equation}
\end{thm}

\begin{proof}
A commutative monoid in $Sp^{\sym}(\M_{\bullet}(S),T^{\wedge 2})$ gives a commutative monoid
in $SH(S)$ by Theorem \ref{T:comparison}.  When $m$ and $n$ are even, the diagram
in $Sp^{\sym}(\M_{\bullet}(S),T^{\wedge 2})$ corresponding to \eqref{E:MSLm.wedge.MSLn}
commutes by formal arguments.  When say $m$ is even and $n$ is odd, the diagram
\[
\xymatrix @M=5pt @C=50pt {
\Sigma_{T}^{\infty}\MSL_{m}(-m) \wedge \Sigma_{T}^{\infty}\MSL_{n}(-n)
\wedge \Sigma_{T}^{\infty}T
\ar[r]^-{\Sigma_{T}^{\infty}\mu_{mn}^{SL}\wedge 1} \ar[d]_-{1 \wedge \Sigma_{T}^{\infty}\sigma_{n}}
& \Sigma_{T}^{\infty}\MSL_{m+n}(-m-n)  \wedge \Sigma_{T}^{\infty}T
\ar[d]^-{\Sigma_{T}^{\infty}\sigma_{m+1}}
\\
\Sigma_{T}^{\infty}\MSL_{m}(-m) \wedge \Sigma_{T}^{\infty}\MSL_{n+1}(-n)
\ar[r]^-{\Sigma_{T}^{\infty}\mu_{m,n+1}^{SL}} \ar[d]_-{u_{m} \wedge u_{n+1}}
& \Sigma_{T}^{\infty}\MSL_{m+n+1}(-m-n) \ar[d]^-{u_{m+n+1}}
\\
\MSL \wedge \MSL(1) \ar[r]^-{\mu^{SL}}
&
\MSL(1)
}
\]
commutes because $m$ and $n+1$ are even. One may desuspend.  The other cases are similar.
\end{proof}


\section{Special linear orientations}

%
%

We now investigate the relationship between special linear orientations
on a ring cohomology theory $E$, as defined in
\cite[Definition 3.1]{Panin:2010aa} and homomorphisms $\varphi \colon \MSL \to A$
of commutative monoids in $SH(S)$.

A \emph{special linear vector bundle} over $X$ is a pair $(E,\lambda)$ with $E \to X$ a vector bundle
and $\lambda \colon \OO_{X} \cong \det E$ an isomorphism of line bundles.
An \emph{isomorphism} $\phi \colon (E,\lambda) \cong (E',\lambda')$ of special linear vector bundles
is an isomorphism $\phi \colon E \cong E'$ of vector bundles such that
$(\det \phi) \circ \lambda = \lambda'$.

\begin{defn}
\label{D:SL.orientation}
A \emph{special linear orientation} on a bigraded $\epsilon$-commutative
ring cohomology theory $A^{*,*}$ on $\SmOp/S$
is a rule which assigns to every special linear vector bundle $(E, \lambda)$ of rank $n$ over an $X$ in
$\Sm/S$ a class $\thom(E,\lambda) \in A^{2n,n}(E,E-X)$
satisfying the following conditions:

\begin{enumerate}
\item For an isomorphism $f \colon (E,\lambda) \cong (E_{1},\lambda_{1})$
we have $\thom(E,\lambda) = f^{*}\thom(E_{1},\lambda_{1})$.

\item For $u \colon Y \to X$
we have $u^{*}\thom(E,\lambda) = \thom(u^{*}(E,\lambda))$ in $A^{2n,n}(u^{*}E,u^{*}E - Y)$.

\item The maps ${-} \cup \thom(E,\lambda) \colon A^{*,*}(X) \to A^{*+2n,*+n}(E,E-X)$ are isomorphisms.

\item We have
\[
\thom (E_{1} \oplus E_{2}, \lambda_{1} \otimes \lambda_{2})
= q_{1}^{*}\thom(E_{1},\lambda_{1}) \cup q_{2}^{*}\thom(E_{2},\lambda_{2}),
\]
where $q_{1},q_{2}$ are the projections from $E_{1} \oplus E_{2}$
onto its summands.  Moreover, for the zero bundle $\boldsymbol 0 \to \pt$ we have
$\thom(\boldsymbol{0}) = 1_{A} \in A^{0,0}(\pt)$.
\end{enumerate}
The class $\thom(E,\lambda)$ is the \emph{Thom class} of the special linear bundle, and
$e(E,\lambda) = z^{*} \thom(E,\lambda) \in A^{2n,n}(X)$ is its \emph{Euler class}.
\end{defn}

This definition is analogous to the Thom classes theory version of the definition of an orientation
\cite[Definition 3.32]{Panin:2003rz}.

For any $n$ the functor ${-}\wedge T^{\wedge n} \colon SH(S) \to SH(S)$ is a self-equivalence.
So it induces isomorphisms
\[
{-}\wedge T^{\wedge n}  \colon
Hom_{SH(S)}(X,A \wedge S^{p,q} ) \xra{\cong}
Hom_{SH(S)}( X \wedge T^{\wedge n},  A \wedge S^{p,q} \wedge T^{\wedge n})
\]
for any $X$ and $(p,q)$ and
any cohomology theory on $\Sm/S$ defined by a commutative monoid $(A,\mu,e)$ in $SH(S)$.
We also write these isomorphisms as
\[
\Sigma_{T}^{n} \colon A^{p,q}(X)  \xra{\cong}
A^{p+2n,q+n}(X \times \Aff^{n}, X \times (\Aff^{n}-0))
\]
This isomorphism coincides with ${-} \times \Sigma_{T}^{n}1_{A}$.
Thus $A^{*,*}$ automatically has Thom classes for trivial bundles: the pullbacks of $\Sigma_{T}^{n}1_{A}$.

%

\begin{defn}
\label{D:SL.normalized}
A special linear orientation on a bigraded ring cohomology theory $A^{*,*}$ on $\SmOp/S$ which is
representable by a commutative monoid in $SH(S)$ is \emph{normalized} if
\begin{enumerate}
\addtocounter{enumi}{4}
\item
for the trivial
line bundle $\Aff^{1} \to \pt$ we have $\thom(\Aff^{1},1) = \Sigma_{T} 1_{A} \in A^{2,1}(\Aff^{1},\Aff^{1}-0)$.
\end{enumerate}
\end{defn}

From the multiplicativity and functoriality conditions (4) and (2) in the definition of a special linear
orientation one deduces the following result.

\begin{lem}
\label{L:triv.Thom}
Suppose $A^{*,*}$ is a bigraded ring cohomology theory on $\SmOp/S$ representable by a commutative
monoid in $SH(S)$ with a normalized special linear orientation.   For $X \in \Sm/S$ let
$(\OO_{X}^{\oplus n},\lambda_{n})$ be the trivial special linear bundle of rank $n$ over $X$.
Then $\thom(\OO_{X}^{\oplus n},\lambda_{n})$ is the pullback to $X$ of $\Sigma_{T}^{n}1_{A}$,
and
\[
{-} \cup \thom(\OO_{X}^{\oplus n},\lambda_{n})
\colon A^{*,*}(X)  \xra{\cong}
A^{*+2n,*+n}(X \times \Aff^{n}, X \times (\Aff^{n}-0))
\]
is an isomorphism.
\end{lem}


Now suppose $\varphi \colon \MSL \to A$ is a morphism
in $SH(S)$.
We associate to $\varphi$ and a special linear bundle $(E,\lambda)$ of rank $n$ over an
$X$ in $\Sm/S$ a class $\thom^{\varphi}(E,\lambda)$ defined as follows.  By assumption
the scheme $X$ admits an ample family of line bundles.  So there exists an
affine bundle $f \colon Y \to X$ with $Y$ an affine scheme.  Then for some $p$ there exist global sections
$s_{1},\dots,s_{np}$ of $f^{*}E^{\vee}$ generating $f^{*}E^{\vee}$.  The data
$(f^{*}E, s_{1}, \dots, s_{np})$ determine a morphism $\psi \colon Y \to Gr(n,np)$, and the data
$(\psi,f^{*}\lambda)$ determine a morphism $\widetilde \psi \colon Y \to \GrSL(n,np)$.
We have $\widetilde\psi^{*} \TSL{n,np} \cong f^{*}E$.  We deduce maps
\begin{equation}
\label{E:Th.E.to.MSLn}
\Th E \xla[\sim\text{mot}]{\overline f} \Th f^{*}E
\,\cong\, \Th \widetilde\psi^{*} \TSL{n,np}
\xra{\overline{\psi}}
\Th \TSL{n,np}
\end{equation}
of pointed motivic spaces, which can be composed with the maps
\begin{equation}
\label{E:MSLn.to.A}
\Th \TSL{n,np} \xra{\text{inclusion}}
 \MSL_{n} \xra{u_{n}} \MSL{} \wedge T^{\wedge n} \xra{1 \wedge \varphi} A \wedge T^{\wedge n}.
\end{equation}
in $SH(S)$.
The composition of \eqref{E:Th.E.to.MSLn} and \eqref{E:MSLn.to.A} gives a class
\[
\thom^{\varphi}(E,\lambda) \in Hom_{SH(S)}(\Th E, A \wedge T^{\wedge n}) = A^{2n,n}(E,E-X).
\]

\begin{lem}
\label{L:well.defined.SL}
The classes $\thom^{\varphi}(E,\lambda)$ depend only on the special linear bundle $(E,\lambda)$
and the morphism  $\varphi \colon \MSL \to A$ in $SH(S)$.
\end{lem}

\begin{proof}
First suppose $f$ fixed.  Let $(s_{1},\dots,s_{np})$ and $(t_{1},\dots,t_{nq})$ be two families
of sections generating $f^{*}E^{\vee}$ with $p \geq q$.
There are $\Aff^{1}$-homotopies between the morphisms $\Th f^{*}E \to \MSL_{n}$ in $\M_{\bullet}(S)$
defined by the family $(s_{1},\dots, s_{np})$, the family $(s_{1},\dots,s_{np},t_{1},\dots,t_{nq})$,
the family $(t_{1},\dots,t_{nq},0,\dots,0,t_{1},\dots,t_{nq})$, and the family
$(t_{1},\dots,t_{nq})$.  So we get the same morphism $\Th f^{*}E \to \MSL_{n}$ in $H_{\bullet}(S)$
and the same morphism $\Th E \to A \wedge T^{\wedge n}$ in $SH(S)$.

Now suppose given a second affine bundle $g \colon Z \to X$ with $Z$ affine
and sections $(u_{1},\dots,u_{nr})$ generating $g^{*}E^{\vee}$.  Let $g' \colon Y \times_{X} Z \to Y$
and $f' \colon Y \times_{X} Z \to Z$ be the projections.  The morphisms $\Th E \to \MSL_{n}$
in $H_{\bullet}(S)$ defined by $f$ and $(s_{1},\dots,s_{np})$, by $g'f$ and $(g'^{*}s_{1}, \dots,
g'^{*}s_{np})$, by $f'g$ and $(f'^{*}u_{1},\dots,f'^{*}u_{nr})$ and by $g$ and $(u_{1},\dots,u_{nr})$
are then the same.  So we again get the same morphism $\Th E \to A\wedge T^{\wedge n}$ in $SH(S)$.
\end{proof}

\begin{thm}
\label{T:thom.phi}
For a homomorphism $\varphi \colon \MSL \to A$ of commutative monoids in $SH(S)$, the classes
$\thom^{\varphi}(E,\lambda)$ define a normalized special linear orientation on the bigraded ring
cohomology theory $A^{*,*}$ on
$\SmOp(S)$.
\end{thm}

In particular the identity homomorphism induces a normalized special linear orientation on
$\MSL^{*,*}$.

\begin{proof}
The functoriality conditions (1) and (2) follow easily from the construction of the
classes $\thom^{\varphi}(E,\lambda)$.  The multiplicativity condition (4) holds because
of Theorem \ref{T:MSL.structure} and because $\varphi$ is a homomorphism of monoids.
The normalization condition (5) holds because $\thom^{\varphi}(\Aff^{1},1)$ and $\Sigma_{T}1_{A}$
are both equal to the composition
\[
T \xra{e_{1}^{SL}} \MSL_{1} \xra{u_{1}} \MSL{} \wedge T\xra{\varphi \wedge 1} A \wedge T.
\]
The isomorphism condition (3) holds for trivial special linear bundles because of
the normalization condition and Lemma \ref{L:triv.Thom}.
It then holds for general special linear bundles by a Mayer-Vietoris argument
because special linear bundles are locally trivial in the Zariski topology.
\end{proof}

Now suppose that $M$ and $A$ are (symmetric) $T$-spectra.  Then we have an
inverse system of abelian groups
\begin{equation}
\label{E:inverse.system}
\cdots \to A^{2n+2,n+1}(M_{n+1}) \xra{\alpha_{n+1}} A^{2n,n}(M_{n}) \to \cdots \to
A^{0,0}(M_{0})
\end{equation}
where the map $\alpha_{n}$ associates to the map $v \colon M_{n+1} \to A \wedge T^{\wedge n+1}$
in $SH(S)$ the composition
\[
M_{n} \xra{\sigma_{n}^{*}} \Omega_{T}M_{n+1} \xra{v} \Omega_{T}(A \wedge T^{\wedge n+1})
\cong A \wedge T^{\wedge n}
\]
in $SH(S)$.  There is a similar inverse system
\begin{equation}
\label{E:inverse.system.2}
\cdots \to A^{4n+4,2n+2}(M_{n+1} \wedge M_{n+1}) \to A^{4n,2n}(M_{n} \wedge M_{n}) \to \cdots \to
A^{0,0}(M_{0} \wedge M_{0}).
\end{equation}

For the following theorem see for example
\cite
[Corollaries 3.4 and 3.5]
{Panin:2009aa}.

\begin{thm}
\label{T:inverse.lim}
For any \parens{symmetric} $T$- or $T^{\wedge 2}$-spectra $M$ and $A$  we have exact sequences
of abelian groups
\begin{gather*}
0 \to \sideset{}{^{1}}\varprojlim A^{2n-1,n}(M_{n}) \to Hom_{SH(S)}(M,A) \to
\varprojlim A^{2n,n}(M_{n}) \to 0,
\\
0 \to \sideset{}{^{1}}\varprojlim A^{4n-1,2n}(M_{n} \wedge M_{n}) \to Hom_{SH(S)}(M \wedge M,A) \to
\varprojlim A^{4n,2n}(M_{n} \wedge M_{n}) \to 0.
\end{gather*}
\end{thm}

This theorem is actually a special case of the following result \cite[Lemma 3.3]{Panin:2009aa}.

\begin{thm}
\label{T:hocolim}
Let $E = \hocolim_{i \in \NN} E^{\angles{i}}$ be a
sequential homotopy colimit of $T$-spectra.
Then for any $T$-spectrum $A$ and any $(p,q)$ we have an exact sequence
of abelian groups
\[
0 \to \sideset{}{^{1}}\varprojlim_{i \in\NN} A^{p-1,q}(E^{\angles{i}}) \to A^{p,q}(E) \to
\varprojlim\limits_{i\in\NN} A^{p,q}(E^{\angles{i}}) \to 0.
\]
\end{thm}

We wish to apply Theorem \ref{T:inverse.lim} when $A$ is a commutative monoid in $SH(S)$
with a normalized special linear orientation on $A^{*,*}$ and when $M$ is a commutative monoid
isomorphic to $\MSL$ in $SH(S)$.  (Note that the exact sequences depend on the levelwise weak
equivalence class of $M$, which is a finer invariant than its isomorphism class in $SH(S)$.)
However, the special linear orientation
provides Thom classes for special linear bundles over finite-dimensional smooth schemes
and not over the infinite-dimensional ind-schemes $BSL_{n}$.
So the orientation does not provide us with classes in
the $A^{2n,n}(\MSL_{n})$.  But we can solve this problem as follows.

For each $n$ and $p$ write
\[
MSL_{n}^{\angles{p}} = \Th \TSL{n,np}.
\]
For $n=0$ this is $MSL_{0}^{\angles{p}} = \pt_{+}$.
The actions of $\sym_{n}$ on $\MSL_{n}$ and the structural maps $e_{n}^{SL}$ and $\mu_{mn}^{SL}$
constructed in the previous section are colimits of actions and structural maps
\begin{gather*}
\sym_{n} \times MSL_{n}^{\angles{p}} \to MSL_{n}^{\angles{p}},
\\
e_{n}^{\angles{p}} \colon T^{\wedge n} \to MSL_{n}^{\angles{p}},
\\
\mu_{mn}^{\angles{p}} \colon MSL_{m}^{\angles{p}} \wedge MSL_{n}^{\angles{p}}
\to MSL_{m+n}^{\angles{p}}.
\end{gather*}
We thus get a direct system of commutative $T$-monoids
\[
MSL^{\angles{1}} \to MSL^{\angles{2}} \to \cdots \to MSL^{\angles{p}} \to \cdots
\]
whose colimit is $\MSL$.  We can now define a ``diagonal'' commutative $T$-monoid
$\MSL^{\finite}$ with spaces
\[
\MSL^{\finite}_{n} = MSL_{n}^{\angles{n}}
\]
with the actions $\sym_{n} \times MSL_{n}^{\angles{n}} \to MSL_{n}^{\angles{n}}$ and unit maps
$e_{n}^{\angles{n}}$ given above and with multiplication maps the compositions
\[
\mu_{mn}^{\finite} \colon
MSL_{m}^{\angles{m}} \wedge MSL_{n}^{\angles{n}} \xra{\text{inclusion}}
MSL_{m}^{\angles{m+n}} \wedge MSL_{n}^{\angles{m+n}} \xra{\mu_{mn}^{\angles{m+n}}}
MSL_{m+n}^{\angles{m+n}}.
\]
A cofinality argument now gives the nontrivial part of the following result.

\begin{thm}
The inclusion $\MSL^{\finite} \hra \MSL$ defines a homomorphism of commutative monoids
in the category of symmetric $T$-spectra which is a motivic stable weak equivalence.
\end{thm}

Thus the inclusion becomes an isomorphism of commutative monoids in $SH(S)$.  So
Theorem \ref{T:inverse.lim} gives us an exact sequence
\begin{equation}
\label{E:exact.seq}
0 \to \varprojlim\nolimits^{1} A^{2n-1,n}(MSL_{n}^{\angles{n}}) \to
Hom_{SH(S)}(\MSL,A) \to \varprojlim A^{2n,n}(MSL_{n}^{\angles{n}}) \to 0
\end{equation}
and a similar exact sequence for $Hom_{SH(S)}(\MSL \wedge \MSL,A)$.

\begin{thm}
\label{T:univ.SL}
Suppose $(A,\mu_{A},e_{A})$ is a commutative monoid in $SH(S)$ with a normalized special linear
orientation on $A^{*,*}$ given by Thom classes $\thom(E,\lambda)$.  Then there exists a
morphism $\varphi \colon \MSL \to A$ in $SH(S)$ such that
$\thom^{\varphi}(E,\lambda) = \thom(E,\lambda)$
for all special linear bundles over all $X$ in $\Sm/S$.   This $\varphi$ is unique modulo the subgroup
\[
\sideset{}{^{1}}\varprojlim A^{2n-1,n}(MSL_{n}^{\angles{n}}) \subset Hom_{SH(S)}(\MSL,A).
\]
It satisfies $\varphi(e_{\MSL}) = e_{A}$.
The obstruction
\(
\varphi \circ \mu_{\MSL} - \mu_{A} \circ (\varphi \wedge \varphi)
\)
to $\varphi$ being a
homomorphism of monoids lies in the subgroup
\[
\sideset{}{^{1}}\varprojlim A^{4n-1,2n}(MSL_{n}^{\angles{n}} \wedge MSL_{n}^{\angles{n}})
\subset Hom_{SH(S)}(\MSL \wedge \MSL,A).
\]
\end{thm}

\begin{proof}
For every $n$ and $p$  the tautological special linear bundle
$(\TSL{n,np}, \lambda_{n,np})$
over the scheme $\GrSL(n,np)$
has a Thom class, which we will abbreviate to
$\thom_{n,np}
\in A^{2n,n}(MSL_{n}^{\angles{p}})$.
Pullback along the inclusion $MSL_{n}^{\angles{p-1}} \hra MSL_{n}^{\angles{p}}$ sends
\begin{equation}
\label{E:th(n,np).1}
\thom_{n,np} \mapsto \thom_{n,n(p-1)}.
\end{equation}
Pullback along the bonding map $MSL_{n-1}^{\angles{p}} \wedge T \to MSL_{n}^{\angles{p}}$
induced by $e_{1}^{\angles{p}}$ and $\mu_{n-1,1}^{\angles{p}}$ sends
\begin{equation}
\label{E:th(n,np).2}
\thom_{n,np} \mapsto
\thom_{n-1,(n-1)p} \times \thom(\Aff^{1},1) =
\Sigma_{T}\thom_{n-1,(n-1)p}.
\end{equation}
%
So as $n$ and $p$ vary, we get an element
\[
\bar\varphi = (
\thom_{n,np}
)
_{n,p}
\in \varprojlim\limits_{n,p} A^{2n,n}(MSL_{n}^{\angles{p}})
= \varprojlim\limits_{n} A^{2n,n}(MSL_{n}^{\angles{n}}).
\]
Let $\varphi \in Hom_{SH(S)}(\MSL,A)$ be an element mapping onto $\bar\varphi$ under the
surjection in the exact sequence \eqref{E:exact.seq}.

The image of $\varphi$ under the composition
\[
Hom_{SH(S)}(\MSL,A) \to \varprojlim A^{2n,n}(MSL_{n}^{\angles{n}}) \to
A^{2n,n}(MSL_{n}^{\angles{n}})
\]
is the composition
\[
\Th \TSL{n,n^{2}} = MSL_{n}^{\angles{n}} \xra{u_{n}} \MSL^{\finite}{}  \wedge T^{\wedge n}
\xra{\sim}
\MSL{}  \wedge T^{\wedge n} \xra{\varphi \wedge 1} A \wedge T^{\wedge n}
\]
which is the $\thom^{\varphi}(\TSL{n,n^{2}},\lambda_{n,n^{2}})$ defined
by \eqref{E:Th.E.to.MSLn}--\eqref{E:MSLn.to.A}.
%
%
%
Thus we have
$\thom^{\varphi}(E,\lambda) = \thom(E,\lambda)$ for
$(E,\lambda) = (\TSL{n,n^{2}}, \lambda_{n,n^{2}})$.
The Thom classes for the $(\TSL{n,n^{2}}, \lambda_{n,n^{2}})$ determine the
Thom classes for all $(\TSL{n,np}, \lambda_{n,np})$ by
formulas \eqref{E:th(n,np).1}--\eqref{E:th(n,np).2}.
These in turn determine the Thom classes for all $(E,\lambda)$ by formulas
\eqref{E:Th.E.to.MSLn}--\eqref{E:MSLn.to.A}.
So we have $\thom^{\varphi}(E,\lambda) = \thom(E,\lambda)$ for
all special linear bundles.

Similarly for $\psi \colon \MSL \to A$ we have
$\thom^{\psi}(E,\lambda) = \thom^{\varphi}(E,\lambda)$ for all special linear bundles
if and only if $\psi$ and $\varphi$ have the same image in
$\varprojlim A^{2n,n}(MSL_{n}^{\angles{n}})$.  This happens if and only if
$\psi - \varphi$ is in the kernel, which is the first $\varprojlim^{1}$ of the statement of the theorem.

By construction $e_{\MSL}$ is the canonical map
$\Sigma_{T}^{\infty} \pt_{+} = \Sigma_{T}^{\infty} \MSL_{0} \to \MSL$.
Therefore we have $\varphi(e_{\MSL}) = \thom_{0,0} = \thom(\boldsymbol{0}) = e_{A} \in A^{0,0}(\pt)$
as declared.

%
By multiplicativity and functoriality we have an equality
%
%
\[
\thom_{n,n^{2}} \times \thom_{n,n^{2}} =
\thom(b_{1}^{*}\TSL{n,n^{2}} \oplus p_{2}^{*}\TSL{n,n^{2}},b_{1}^{*}\lambda_{n,n^{2}} \otimes p_{2}^{*}\lambda_{n,n^{2}})
=
\mu_{nn}^{\finite *}
\thom_{2n,4n^{2}}
\]
of members of
$A^{4n,2n}(MSL_{n}^{\angles{n}} \wedge MSL_{n}^{\angles{n}})$.  This equality means that the
outer perimeter of the diagram
\[
\xymatrix @M=5pt {
\Sigma_{T}^{\infty} MSL_{n}^{\angles{n}} (-n)
\wedge
\Sigma_{T}^{\infty} MSL_{n}^{\angles{n}} (-n)
\ar[r]^-{\mu_{nn}^{\finite}}
\ar[d]^-{u_{n} \wedge u_{n}}
\ar@/_25mm/[ddd]_-{\thom_{n,n^{2}} \wedge \thom_{n,n^{2}}}
&
\Sigma_{T}^{\infty} MSL_{2n}^{\angles{2n}} (-2n)
\ar[d]^-{u_{2n}}
\ar@/^25mm/[ddd]^-{\thom_{2n,4n^{2}}}
\\
\MSL^{\finite} \wedge \MSL^{\finite}
\ar[r]^-{\mu_{\MSL}^{\finite}}
\ar[d]
^-{\text{inclusion}}
&
\MSL^{\finite}
\ar[d]
^-{\text{inclusion}}
\\
\MSL \wedge \MSL
\ar[r]^-{\mu_{\MSL}}
\ar[d]^-{\varphi \wedge \varphi}
&
\MSL
\ar[d]^-{\varphi}
\\
A \wedge A \ar[r]^-{\mu_{A}}
& A
}
\]
commutes.  The half-circles commute by the previous calculations, and the top two squares commute.
Therefore we have
\[
\bigl( \varphi \circ \mu_{\MSL} - \mu_{A} \circ (\varphi \wedge \varphi) \bigr)
\circ
\text{inclusion}
\circ (u_{n} \wedge u_{n}) = 0
\]
for all $n$.  So the image of the obstruction class
$\varphi \circ \mu_{\MSL} - \mu_{A} \circ (\varphi \wedge \varphi)$
under the surjection
\[
Hom_{SH(S)}(\MSL \wedge \MSL,A)
\lra
\varprojlim A^{4n,2n}(MSL_{n}^{\angles{n}} \wedge MSL_{n}^{\angles{n}})
\lra 0
\]
vanishes. Therefore the obstruction class lies in the kernel, which is the second $\varprojlim^{1}$
of the statement of the theorem.
\end{proof}

\section{The symmetric
{$T^{\wedge 2}$}{T\^{ }2}-spectrum
{$\MSp$}{MSp}}

We now define the commutative $T^{\wedge 2}$-monoid and
symmetric $T^{\wedge 2}$-spectrum $\MSp$.

We write the standard symplectic form on the trivial vector bundle of rank $2n$ as
\[
\omega_{2n} =
\begin{pmatrix}
\begin{matrix}
0&1 \\ -1&0
\end{matrix}
&& 0 \\
& \ddots & \\
0 &&
\begin{matrix}
0&1 \\ -1&0
\end{matrix}
\end{pmatrix}
\]
From the symplectic isometry
$(\OO_{S}^{\oplus 2n},\omega_{2n}) \cong (\OO_{S}^{\oplus 2},\omega_{2})^{\oplus n}$
we see that the action of $\sym_{n}$ given by permutations of the $n$ orthogonal direct summands
$(\OO_{S}^{\oplus 2},\omega_{2})$
gives an embedding $\sym_{n} \to Sp_{2n}$.  Hence $Sp_{2n}$-actions restrict to
$\sym_{n}$-actions.

In \cite{Panin:2010fk} we defined the quaternionic Grassmannian $HGr(r,n)$ as the open subscheme
of the Grassmannian $Gr(2r,2n)$ parametrizing rank $2r$ subspaces of $\OO_{S}^{\oplus 2n}$ on which
the restriction of $\omega_{2n}$ is nondegenerate.  The restriction of the tautological subbundle
over the Grassmannian is the tautological symplectic subbundle $\TSp{r,n}$.  It is equipped with
the symplectic form $\phi_{r,n} = \omega_{2n} |_{\TSp{r,n}}$.
For
$r=1$ we write $HP^{n} = HGr(1,n+1)$ and $HP^{\infty} = \colim_{n} HP^{n}$.

To construct $\MSp$ we look at the particular schemes
$HGr(n,np) = HGr(n, (\OO_{S}^{\oplus 2n}, \omega_{2n})^{\oplus p})$.  Each has a natural action
of $Sp_{2n}$ induced by the diagonal action of $Sp_{2n}$ on the $p$ summands of
$(\OO_{S}^{\oplus 2n}, \omega_{2n})^{\oplus p}$.  The vector bundles
$\TSp{n,np} \to HGr(n,np)$ and the inclusions $HGr(n,np) \to HGr(n,np+n)$
are $Sp_{2n}$-equivariant.  We set
\begin{gather*}
BSp_{2n} = \colim\nolimits_{p \in \NN} HGr(n,np), \\
\TSp{n,n\infty} = \colim\nolimits_{p \in \NN} \TSp{n,np}, \\
\MSp_{2n} = \colim\nolimits_{p \in \NN} \Th \TSp{n,np}.
\end{gather*}

As with $\MGL$ and $\MSL$ the isomorphisms
\[
(\OO_{S}^{\oplus 2m}, \omega_{2m}) \oplus (\OO_{S}^{\oplus 2n}, \omega_{2n})
\cong (\OO_{S}^{\oplus 2m+2n}, \omega_{2m+2n})
\]
and the direct sum induce $(Sp_{2m} \times Sp_{2n})$-equivariant maps
\begin{equation}
\label{E:monoid.Sp}
\begin{gathered}
\oplus \colon BSp_{2m} \times BSp_{2n} \to BSp_{2m+2n},
\\
\mu^{Sp}_{mn} \colon \MSp_{2m} \wedge \MSp_{2n} \to \MSp_{2m+2n}.
\end{gathered}
\end{equation}

Each $HGr(n,np)$ is pointed by the point corresponding to the symplectic subbundle
which is the first direct summand
$(\OO_{S}^{\oplus 2n}, \omega_{2n}) \oplus 0^{\oplus p-1} \subset
(\OO_{S}^{\oplus 2n}, \omega_{2n}) ^{\oplus p}$.  In the colimit this yields
points $z_{2n} \colon \pt \to BSp_{2n}$.  The point $z_{2n}$ is fixed by the
$Sp_{2n}$-action.  The action of $Sp_{2n}$ on the fiber of $\TSp{n,n\infty}$
over $z_{2n}$ is the standard representation of $Sp_{2n}$.  The inclusion of the
fiber induces an inclusion of Thom spaces
\begin{equation}
\label{E:e.Sp}
e_{2n}^{Sp} \colon T^{\wedge 2n} \to \MSp_{2n}
\end{equation}
which is $Sp_{2n}$-equivariant.  The action of the subgroup $\sym_{n} \subset Sp_{2n}$
on $T^{\wedge 2n} = (T^{\wedge 2})^{\wedge n}$ permutes the $n$ factors $T^{\wedge 2}$.

The spaces $\MSp_{2n}$ with the actions and structural maps verify the axioms of a
commutative $T^{\wedge 2}$-monoid.

\begin{defn}
The \emph{algebraic symplectic cobordism spectrum} $\MSp$ is the
commutative monoid in the category of symmetric $T^{\wedge 2}$-spectra associated
to the commutative $T^{\wedge 2}$-monoid composed of
the spaces $\MSp_{2n}$, the actions $\sym_{n} \times \MSp_{2n} \to \MSp_{2n}$
the maps
$e_{0}^{Sp} \colon \pt_{+} \to \MSp_{0}$ and $e_{2}^{Sp} \colon T^{\wedge 2} \to \MSp_{2}$
and the maps $\mu_{mn}^{Sp} \colon \MSp_{2m} \wedge \MSp_{2n} \to \MSp_{2m+2n}$.
\end{defn}

This $\MSp$ defines a commutative monoid in $SH(S)$ by Theorem \ref{T:comparison}.

\section{Quaternionic Grassmannian bundles}

We review the geometry of quaternionic projective bundles and Grassmannian bundles studied
in \cite[\S\S 3--5]{Panin:2010fk}.  We then translate some of the results into a more motivic language.

Given $(E,\phi)$ a symplectic bundle of rank $2n$ over a scheme $X$ and an integer $0 \leq r \leq n$,
there is a quaternionic Grassmannian bundle $p \colon HGr(r,E,\phi) \to X$ whose fiber over
$x \in X$ is the quaternionic Grassmannian parametrizing $2r$-dimensional subspaces of $E_{x}$
on which $\phi_{x}$ is nondegenerate.
We write $\shf U_{r,E} \subset p^{*}E$ for the tautological rank
$2r$ subbundle over $HGr(r,E,\phi)$.
Morphisms $f \colon Y \to HGr(r,E,\phi)$ are in bijection
with pairs $(g,U)$ where $g \colon Y \to X$ is a morphism and $U \subset g^{*}(E,\phi)$ is a symplectic
subbundle of rank $2r$ over $Y$.

Since $\shf U_{r,E}$ is a subbundle on which the symplectic form is fiberwise nondegenerate, it has
an orthogonal complement such that $\shf U_{r,E} \oplus \shf U_{r,E}^{\perp} = p^{*}E$.  The symplectic
subbundle $\shf U_{r,E}^{\perp } \subset p^{*}E$ classifies an isomorphism
\begin{equation}
\label{E:perp}
HGr(r,E,\phi) \xra{\cong} HGr(n-r,E,\phi).
\end{equation}

Now let $(F,\psi) = (\OO_{X}^{\oplus 2},\omega_{2}) \oplus (E,\phi)$.  We have a natural embedding
\begin{equation}
\label{E:embed.1}
HGr(r,E,\phi) \hra HGr(r,F,\psi)
\end{equation}
classified by the symplectic subbundle $0 \oplus \shf U_{r,E} \subset \OO_{X}^{\oplus 2} \oplus E = F$.
The normal bundle of this embedding can be naturally identified with the vector bundle
$N = \shf Hom(\shf U_{r,E},\OO_{X}^{\oplus 2})$ over $HGr(r,E,\phi)$.  This bundle is a direct sum
decomposition $N = N^{+} \oplus N^{-}$ where
\begin{align*}
N^{+} & = \shf Hom(\shf U_{r,E}, \OO_{X} \oplus 0),
&
N^{-} & = \shf Hom(\shf U_{r,E}, 0 \oplus \OO_{X}).
\end{align*}
The basic result concerning the geometry of the closed embedding \eqref{E:embed.1} is the following.


\begin{thm}[\protect{\cite[Theorem 4.1]{Panin:2010fk}}]
\label{T:normal.geom}
\parens{a} The normal bundle of the embedding \eqref{E:embed.1} has
a canonical open embedding $\nu \colon N \hra Gr(2r,F)$.  The zero section is sent
identically onto $HGr(r,E,\phi)$.

\parens{b} We have $\nu(N^{+}) = HGr(r,F,\psi) \cap \Gr_{S}(2r,\OO_{X}\oplus 0\oplus E)$.
Consequently
$\nu(N^{+}) \subset HGr(r,F,\psi)$ is a closed subscheme, as is $\nu(N^{-}) \subset HGr(r,F,\psi)$.

\parens{c} There are natural isomorphisms of vector bundles
$N^{+} \cong N^{-} \cong \shf U_{r,E}^{\vee} \cong \shf U_{r,E}$.

\parens{d} There is a natural section $s_{+}$ of $\shf U_{r,F}$
intersecting the zero section transversally in $N^{+}$ and similarly for $N^{-}$.

\parens{e} Let $\pi_{+} \colon N^{+} \to HGr(r,E,\phi)$ be the structural map.  Then
$\pi_{+}^{*}(\shf U_{r,E},\phi|_{\shf U_{r,E}})$ is isometric to
$(\shf U_{r,F},\psi|_{\shf U_{r,F}})|_{N^{+}}$ and similarly for $N^{-}$.
\end{thm}

The second basic result about the geometry of symplectic Grassmannian bundles involves the
following embeddings
\begin{equation}
\label{E:embed.2}
\xymatrix @M=5pt @C=40pt {
HGr(r-1,E,\phi) \ar@{^{(}->}[r]^-{\sigma}_-{\text{closed}} &
HGr(r,F,\psi) - \nu(N^{+}) \ar@{^{(}->}[r]^-{\text{open}} & HGr(r,F,\psi).
}
\end{equation}
The composition is the closed embedding classified by the symplectic subbundle
$\OO^{\oplus 2} \oplus \shf U_{r-1,E} \subset \OO^{\oplus 2} \oplus E = F$.

\begin{thm}[\protect{\cite[Theorems 5.1 and 5.2]{Panin:2010fk}}]
\label{T:affine.bundles}
There are morphisms over $X$
\[
HGr(r,F,\psi) - \nu(N^{+}) \xleftarrow{g_{1}} Y_{1} \xleftarrow{g_{2}}
Y_{2} \xrightarrow{q} HGr(r-1,E,\phi)
\]
with $g_{1}$ an $\Aff^{2r-1}$-bundle, $g_{2}$ an $\Aff^{2r-2}$-bundle, and
$q$ an $\Aff^{4n+1}$-bundle.  Moreover,
there is a section $s$
of $q$ such that the composition
$g_{1}g_{2}s \colon HGr(r-1,E,\phi) \to HGr(r,F,\psi)- \nu(N^{+})$
is the closed embedding $\sigma$ of \eqref{E:embed.2}.
\end{thm}

The section $s$ appears at the end of the proof of \cite[Theorem 5.2]{Panin:2010fk}.  The statement
of the theorem only contains the consequence that $\sigma$ induces isomorphisms of cohomology groups.

These two theorems have the following consequence.

\begin{thm}
\label{T:HGr/HGr}
Let $(E,\phi)$ be a symplectic bundle of rank $2n$ over a smooth $S$-scheme $X$, and let
$(F,\psi) = (\OO_{X}^{\oplus 2}, \omega_{2}) \oplus (E,\phi)$.  For $1 \leq r \leq n$ let
$\shf U_{r,E}$ be the tautological symplectic subbundle over $HGr(r,E,\phi)$.
Let $HGr(r-1,E,\phi) \hra HGr(r,F,\psi)$ be the closed embedding of \eqref{E:embed.2}.
Then there is a canonical zigzag of motivic weak equivalences
\[
\Th \shf U_{r,E}
\xra{\sim} HGr(r,F,\psi)/(HGr(r,F,\psi) - \nu(N^{+}))
\xla{\sim}
HGr(r,F,\psi) / HGr(r-1,E,\phi)
\]
inducing an isomorphism in the motivic unstable homotopy category $H_{\bullet}(S)$.
These isomorphisms commute with the maps induced by inclusions
$(E,\phi) \hra (E,\phi) \oplus (E_{1},\phi_{1})$ of symplectic bundles and with the maps
induced by base changes $Y \to X$.
\end{thm}

\begin{proof}
In the geometry of Theorem \ref{T:normal.geom}
we identify $N$, $N^{+}$ and $N^{-}$ with their images in $Gr(2r,F)$ under the open embedding $\nu$.
Then there
are motivic weak equivalences
\[
\xymatrix @M=5pt @C=20pt {
N/(N-N^{+}) & \ar[l]_-{\sim}^-{\text{excision}}
(N \cap HGr(r,F,\psi))/(N \cap HGr(r,F,\psi) - N^{+}) \ar[d]_-{\sim}^-{\text{excision}}
\\
\Th \shf U_{r,E} \cong
N^{-}/(N^{-}-HGr(r,E,\phi))
\ar[u]_-{\sim}^-{\text{section of vector bundle}}
\ar[ru] \ar[r]_-{\text{inclusion}}
&
HGr(r,F,\psi)/(HGr(r,F,\psi)-N^{+})
}
\]
The arrows not explicitly labeled $\sim$ are nevertheless motivic weak equivalences by the 2-out-of-3
axiom.  The morphisms $g_{1}$, $g_{2}$ and
$q$ of Theorem \ref{T:affine.bundles} are affine bundles, so they, the section $s$,
the composition $\sigma$, and the map
\[
HGr(r,F,\psi)/HGr(r-1,E,\phi) \xra{\sim} HGr(r,F,\psi)/(HGr(r,F,\psi)-N^{+})
\]
of quotient spaces induced by $\sigma$ are all motivic weak equivalences.

The functoriality is straightforward.
\end{proof}

Because of \eqref{E:perp} the quotient appearing in Theorem \ref{T:HGr/HGr} is also isomorphic to
\[
HGr(n-r+1,F,\psi)/HGr(n-r+1,E,\phi).
\]
The case $r=n$ of Theorem \ref{T:HGr/HGr} therefore gives the following result.
The corresponding result for ordinary Grassmannian
bundles is \cite[Proposition 3.2.17(3)]{Morel:1999ab}.

\begin{thm}
Suppose that $(E,\phi)$ is a symplectic bundle of rank $2n$ over a smooth $S$-scheme $X$.
Let $HP(E,\phi) \hra HP(\OO_{X}^{\oplus 2} \oplus E, \omega_{2} \oplus \phi)$
be the natural closed embedding.  Then we have isomorphisms in $H_{\bullet}(S)$
\[
\Th E \cong HGr(n,(\OO_{X}^{\oplus 2}, \omega_{2}) \oplus (E,  \phi))/HGr(n-1,E,\phi) =
HP((\OO_{X}^{\oplus 2}, \omega_{2}) \oplus (E,  \phi)) / HP(E,\phi).
\]
\end{thm}

For $X = \pt$ and trivial $E$ Theorem \ref{T:HGr/HGr} gives the following result.

\begin{thm}
\label{T:BSp/BSp}
There are canonical isomorphisms
\[
\Th \shf U_{HGr(r,n)} \cong HGr(r,1+n)/HGr(r-1,n)
\]
in $H_{\bullet}(S)$.  These isomorphisms are compatible with the inclusions of quaternionic projective spaces.
Therefore we have commutative diagrams of
inclusions and isomorphisms in $H_{\bullet}(S)$
\begin{equation}
\label{E:BSp/BSp}
\vcenter{
\xymatrix @M=5pt @C=40pt {
T^{\wedge 2r} \ar[r]^-{e_{r}^{Sp}} \ar[d]_-{\cong}
& \MSp_{2r} \ar[d]^-{\cong}
\\
HP^{r}/HP^{r-1} \ar[r]_-{\text{\textup{inclusion}}}
& BSp_{2r}/BSp_{2r-2}
}}
\end{equation}
\end{thm}

The motivic weak equivalences of Theorem \ref{T:BSp/BSp} fit into a commutative diagram
\begin{equation}
\label{E:amazing.1}
\vcenter{
\xymatrix @M=5pt @C=30pt {
HGr(r,n) \ar[rr]^-{\text{inclusion}} \ar[d]_-{\text{section map}}
&&
HGr(r,1+n) \ar[d]^-{\text{quotient}}
\\
\Th \shf U_{HGr(r,n)} \ar[r]^-{\sim \, mot}
&
Y_{r,n}
&
HGr(r,1+n)/HGr(r-1,n) \ar[l]_-{\sim \,mot}
}}
\end{equation}
The \emph{section map} is the structure map of the Thom space induced by a section of the vector bundle.
We have a colimit as $n \to \infty$:
\begin{equation}
\label{E:amazing.2}
\vcenter{
\xymatrix @M=5pt @C=30pt {
BSp_{2r} = HGr(r,\infty) \ar[rr]^-{\text{shift}}
\ar[d]_-{\text{section map}}
&&
BSp_{2r} = HGr(r,1+\infty) \ar[d]^-{\text{quotient}}
\\
\MSp_{2r} \ar[r]^-{\sim \, mot}
&
Y_{r,\infty}
&
BSp_{2r}/BSp_{2r-2} \ar[l]_-{\sim \,mot}
}}
\end{equation}
The \emph{shift map} is the map induced by the \emph{shift endomorphism} of
$(\OO^{\oplus 2},\omega_{2})^{\oplus \infty}$ acting on sequences of sections of $\OO$ by
\[
(s_{1},s_{2},s_{3},s_{4},\dots) \mapsto (0,0,s_{1},s_{2},s_{3},s_{4},\dots).
\]

\begin{lem}
There is an $\Aff^{1}$-homotopy of symplectic endomorphisms of
$(\OO^{\oplus 2},\omega_{2})^{\oplus \infty}$ linking the
shift endomorphism of $(\OO^{\oplus 2},\omega_{2})^{\oplus \infty}$ to the identity.
\end{lem}


\begin{proof}
%

Any matrix in any $Sp_{2r}(\ZZ)$ is a product of elementary symplectic matrices.  So there
exists an $M(t) \in Sp_{4}(\ZZ[t])$ with
\begin{align*}
M(0) & =
\begin{pmatrix}
1 & 0 & 0 & 0 \\0 & 1 & 0 & 0 \\0 & 0 & 1 & 0 \\0 & 0 & 0 & 1
\end{pmatrix},
&
M(1) & =
\begin{pmatrix}
0 & 0 & 1 & 0 \\0 & 0 & 0 & 1 \\1 & 0 & 0 & 0 \\0 & 1 & 0 & 0
\end{pmatrix},
\end{align*}
for example
\[
M(t) =
\begin{pmatrix}
1-t^{2} & 0 & -2t+13t^{3}-14t^{5}+4t^{7} & 8t^{2}-12t^{4}+4t^{6}
\\
0 & 1-t^{2} & -2t^{2}+2t^{4} & -t + 2t^{3}
\\
t & 0 & 1-7t^{2}+10t^{4}-4t^{6} & -4t + 8t^{3} -4t^{5}
\\
0 & 2t-t^{3} & 2t-4t^{3}+2t^{5} & 1 - 3t^{2} + 2t^{4}
\end{pmatrix}.
\]

Consider the sequence of infinite matrices
\[
f_{n}(t) =
\begin{pmatrix}
I_{2n-2} & 0 & 0 \\ 0 & M(t) & 0 \\ 0 & 0 & I_{\infty}
\end{pmatrix}.
\]
Then $f_{n}(1)$ acts on  $(\OO^{\oplus 2},\omega_{2})^{\oplus \infty}$ by exchanging the
$n^{\text{th}}$ and $(n+1)^{\text{st}}$ summands $(\OO^{\oplus 2},\omega_{2})$, while
$f_{n}(0)$ is the identity.  The infinite product
\[
F(t) = f_{1}(t) f_{2}(t) f_{3}(t) \cdots
\]
is well-defined
because the first $2n$ columns of $f_{1}f_{2}\cdots f_{N}$ are independent of $N$ for all $N \geq n$.
Each column of $F(t)$ contains only a finite number of nonzero entries.  (More precisely, writing
$F(t) = (a_{ij}(t))$, we have $a_{i,2n-1}(t) = a_{i,2n}(t)= 0$ for $i > 2n+2$.)
The endomorphism $F(t)$ is
preserves the symplectic form $\omega_{2}^{\oplus \infty}$ because the automorphisms $f_{n}(t)$ all do.
Clearly $F(0)$ is the identity.  The finite product $f_{1}(1)f_{2}(1) \cdots f_{N}(1)$ permutes cyclically
the first $N+1$ summands of $(\OO^{\oplus 2},\omega_{2})^{\infty}$ and fixes the others.
The infinite product $F(1)$ is the shift map.
\end{proof}

\begin{thm}
\label{T:MSp=BSp/BSp}
In the motivic unstable homotopy category $H_{\bullet}(S)$ we have a commutative diagram
\[
\xymatrix @M=5pt {
& BSp_{2r} \ar@/_/[ld]_-{\text{\textup{structure map}}} \ar@/^/[rd]^-{\text{\textup{quotient}}}
\\
\MSp_{2r} \ar@{<->}[rr]^-{\cong}_-{\text{\textup{$\Aff^{N}$-bundles and excision}}}
&& BSp_{2r}/BSp_{2r-2}.
}
\]
\end{thm}

\section{The quaternionic projective bundle theorem}
\label{S:projective.bundle}

The most basic form a symplectic orientation is a symplectic Thom structure
\cite[Definition 7.1]{Panin:2010fk}.  The version of the definition for bigraded
$\epsilon$-commutative theories is as follows.

\begin{defn}
\label{D:sp.thom}
A \emph{symplectic Thom structure} on a bigraded $\epsilon$-commutative
ring cohomology theory $(A^{*,*},\partial,\times,1_{A})$ on $\SmOp/S$ is a
rule which assigns to each rank $2$ symplectic bundle $(E,\phi)$ over an $X$
in $\Sm/S$ an element
$\thom(E,\phi) \in A^{4,2}(E,E-X)$ with the following properties:
\begin{enumerate}
\item For an isomorphism $u \colon (E,\phi) \cong (E_{1},\phi_{1})$
one has $\thom(E,\phi) = u^{*}\thom(E_{1},\phi_{1})$.

\item For a morphism $f \colon Y \to X$ with pullback map $f_{E} \colon f^{*}E \to E$
one has $f_{E}^{*}\thom(E,\phi) = \thom(f^{*}E,f^{*}\phi)$.

\item For the trivial rank $2$ bundle $\Aff^{2} \to \pt$ with the symplectic form
$\omega_{2} =  \bigl( \begin{smallmatrix} 0 & 1 \\ -1 & 0 \end{smallmatrix} \bigr)$
the map
\[
{-}\times \thom(\Aff^{2},\omega_{2}) \colon
A^{*,*}(X) \to A^{*+4,*+2}(X \times \Aff^{2},X \times (\Aff^{2}-0))
\]
is an isomorphism for all $X$.
\end{enumerate}
The \emph{Borel class} of $(E,\phi)$ is $b_{1}(E,\phi) = -z^{*} \thom(E,\phi) \in A^{4,2}(X)$ where
$z \colon X \to E$ is the zero section.
\end{defn}

The quaternionic projective bundle theorem is proven in \cite{Panin:2010fk} using the symplectic Thom
structure and not any other version of a symplectic orientation.  It is proven first for trivial bundles.

\begin{thm}[\protect{\cite[Theorem 8.1]{Panin:2010fk}}]
\label{T:H.proj.bdl.1}
Let $(A^{*,*},\partial,\times,1_{A})$ be a bigraded $\epsilon$-commutative ring cohomology theory with
a symplectic Thom structure.
Let $(\mathcal U_{HP^{n}}, \phi_{HP^{n}})$ be the tautological rank $2$
symplectic subbundle over $HP^n$ and
$\zeta = b_{1}(\mathcal U_{HP^{n}}, \phi_{HP^{n}})$ its Borel class.
Then for any $X$ in $\Sm/S$ we have an isomorphism of bigraded rings
\[
A^{*, *}(HP^n \times X) \cong A^{*,*}(X)[\zeta]/(\zeta^{n+1}).
\]
\end{thm}

A Mayer-Vietoris argument gives the more general theorem \cite[Theorem 8.2]{Panin:2010fk}.

\begin{thm}[\protect{Quaternionic projective bundle theorem}]
\label{T:H.proj.bdl.2}
Let $(A^{*,*},\partial,\times,1_{A})$ be a bigraded $\epsilon$-commutative ring cohomology theory with
a symplectic Thom structure.
Let $(E,\phi)$  be a symplectic bundle of rank $2n$ over $X$, let
$(\mathcal U, \phi|_{\mathcal U})$ be the tautological rank $2$
symplectic subbundle over
the quaternionic projective bundle $HP(E, \phi)$, and let
$\zeta = b_{1}(\mathcal U, \phi|_{\mathcal U})$
be its Borel class. Then
we have an isomorphism of bigraded $A^{*,*}(X)$-modules
\[
(1, \zeta, \dots , \zeta^{n-1}) \colon
 A^{*,*}(X) \oplus A^{*,*}(X) \oplus \dots \oplus
A^{*,*}(X) \to A^{*,*}(HP(E, \phi)).
\]
\end{thm}

\begin{defn}
\label{D:PontryaginClasses}
Under the hypotheses of Theorem
\ref{T:H.proj.bdl.2}
there are unique elements
$b_i(E, \phi) \in A^{4i,2i}(X)$ for $i=1,2, \dots , n$
such that
$$\zeta^{n}-b_1(E, \phi)\cup \zeta^{{n-1}} + b_2(E, \phi)\cup \zeta^{{n-2}} - \dots + (-1)^n b_n(E, \phi)=0.$$
The classes
$b_i(E, \phi)$
are called the \emph{Borel classes}
of $(E, \phi)$ with respect to the symplectic Thom structure of the cohomology theory $(A, \partial)$.
For $i > n$ one sets $b_i(E, \phi)$ = 0, and one sets $b_0(E, \phi) = 1$.
\end{defn}

For a rank $2$ symplectic bundle $(E,\phi)$ the classes $b_{1}(E,\phi)$ defined by Definitions
\ref{D:sp.thom}   and \ref{D:PontryaginClasses} coincide.

\begin{cor}
\label{C:pont.triv}
The Borel classes of a trivial symplectic bundle vanish.
\end{cor}

Among the consequences of the quaternionic projective bundle theorem is the symplectic splitting
principle \cite[Theorem 10.2]{Panin:2010fk}.
We used it to prove the Cartan sum formula for Borel classes \cite[Theorem 10.5]{Panin:2010fk}.

\begin{thm}
\label{T:Cartan}
Let $(A^{*,*},\partial,\times,1_{A})$ be a bigraded $\epsilon$-commutative ring cohomology theory with
a symplectic Thom structure.
Suppose $(F, \psi) \cong (E_{1},\phi_{1}) \oplus (E_{2},\phi_{2})$ is an orthogonal direct sum of symplectic bundles over an $X$ in $\Sm/S$.  Then for all $i$ we have
\begin{equation}
\label{E:sum.formula}
b_{i}(F,\psi) = b_{i}(E_{1},\phi_{1}) + \sum_{j=1}^{i-1} b_{i-j}(E_{1},\phi_{1}) b_{j}(E_{2},\phi_{2}) +
b_{i}(E_{2},\phi_{2}).
\end{equation}
\end{thm}

The quaternionic projective bundle theorem also allowed us to compute the cohomology of
quaternionic Grassmannians.  To explain our results we need to recall a number of facts
about symmetric polynomials.
They may be found in  for example \cite[Chap.\ 1, \S\S 1--3]{Macdonald:1995zl}.

Let $\Lambda_{r} \subset \ZZ[x_{1},\dots,x_{r}]$ be the ring of symmetric polynomials in $r$ variables.
Let $e_{i}$ denote the $i^{\text{th}}$ elementary symmetric polynomial, and
$h_{i}$ the $i^{\text{th}}$ \emph{complete symmetric polynomial}, the sum of all the monomials of degree $i$.
Set $e_{0} = h_{0} = 1$
and $e_{i} = h_{i} = 0$ for $i < 0$ and also $e_{i} = 0$ for $i > r$.
We have $\Lambda_{r} = \ZZ[e_{1},\dots,e_{r}]$.
There is a recurrence relation
$h_{m} + \sum_{i=1}^{r} (-1)^{r} e_{i}h_{m-i}  = 0$.
%

Let
\[
\Pi_{r} = \{ \text{partitions $\lambda = (\lambda_{1},\lambda_{2},\dots,\lambda_{r})$
of length $l(\lambda)  \leq r$} \}.
\]
Write $\delta = (r-1,r-2, \dots,1,0)$.
For $\lambda \in \Pi_{r}$ let
$a_{\lambda + \delta} = \det(x_{i}^{\lambda_{j}+r-j})_{1\leq i,j \leq r}$.  Then $a_{\lambda+\delta}$ is
a skew-symmetric polynomial and therefore divisible by the Vandermonde determinant $a_{\delta}$.
The quotient $s_{\lambda} = a_{\lambda+\delta}/a_{\delta}$ is
the \emph{Schur polynomial} for $\lambda$.
It is symmetric of degree $\lvert \lambda \rvert = \sum \lambda_{i}$.
One has $s_{(1^{i})} = e_{i}$ and $s_{(i)} = h_{i}$.
The $a_{\lambda+\delta}$ with $l(\lambda) \leq r$ form a $\ZZ$-basis of the skew-symmetric polynomials
in $r$ variables, so the $s_{\lambda}$ with $l(\lambda) \leq r$ form a $\ZZ$-basis of $\Lambda_{r}$.
Denote by $\lambda'$ the partition dual to $\lambda$.  We have formulas
\begin{equation}
\label{E:schur.det}
s_{\lambda} = \det(e_{\lambda'_{i}-i+j})_{1 \leq i,j \leq m}
= \det(h_{\lambda_{i}-i+j})_{1\leq i,j, \leq r},
\end{equation}
for $m \geq l(\lambda')$ and $r \geq l(\lambda)$.
%
Set
\[
\Pi_{r,n-r} = \{ \text{partitions $\lambda$ of length $l(\lambda) = \lambda'_{1} \leq r$
and with  $\lambda_{1} \leq n-r$} \}
\]
The set $\Pi_{r,n-r}$ has $\binom{n}{r}$ members.
We will use the following results.

\begin{prop}
\label{P:sym.func.1}
The quotient map
\[\ZZ[e_{1},\dots,e_{r}] \to \ZZ[e_{1},\dots,e_{r}]/(h_{n-r+1},\dots,h_{n})
\]
sends
$\{s_{\lambda} \mid \lambda \in \Pi_{r} - \Pi_{r,n-r} \} \mapsto 0$,
and it sends
$\{s_{\lambda} \mid \lambda \in \Pi_{r,n-r} \}$ onto a homogeneous $\ZZ$-basis of the quotient ring.
\end{prop}

\begin{proof}
For $\lambda \in \Pi_{r} - \Pi_{r,n-r}$ the first line of the determinant
$s_{\lambda} = \det(h_{\lambda_{i}-i+j})$ consists of $h_{k}$ with $k \geq n-r+1$.  These are all
in $(h_{n-r+1},\dots,h_{n})$ because of the recurrence relation satisfied by the $h_{k}$.  So
they are sent to $0$ in the quotient.

The rank of the quotient as a $\ZZ$-module is $\prod \deg h_{i} / \prod \deg e_{i} =
\binom{n}{r}$.  Since this is the same as the cardinality of $\Pi_{r,n-r}$, and since the images of
the $s_{\lambda}$ with $\lambda \in \Pi_{r,n-r}$ generate the quotient as a $\ZZ$-module,
they form a $\ZZ$-basis of the quotient.
\end{proof}

Let the $\overline e_{i}$ and the $\overline h_{i}$ be, respectively, the elementary and complete
symmetric polynomials in $r-1$ variables.  The natural quotient map sends
$e_{i} \mapsto \overline e_{i}$ for $i < r$ and $e_{r} \mapsto 0$, while it sends
$h_{i} \mapsto \overline h_{i}$ for all $i$.

\begin{prop}
\label{P:sym.func.2}
The kernel of the quotient map
\[
\ZZ[e_{1},\dots,e_{r}]/(h_{n-r+1},\dots,h_{n})
\to
\ZZ[\overline e_{1},\dots, \overline e_{r-1}]/(\overline h_{n-r+1},\dots, \overline h_{n-1})
\to 0
\]
is the image of the injection
\[
0 \to
\ZZ[e_{1},\dots,e_{r}]/(h_{n-r},\dots,h_{n-1})
\xra{ e_{r} \times {-}}
\ZZ[e_{1},\dots,e_{r}]/(h_{n-r+1},\dots,h_{n}).
\]
\end{prop}

\begin{proof}
The kernel is the free $\ZZ$-module with basis
$\{ s_{\lambda} \mid \lambda \in \Pi_{r,n-r} - \Pi_{r-1,n-r} \}$.  These are the $\lambda$
with $\lambda_{r} \geq 1$ and thus $\lambda'_{1} = r$.  The formula
$s_{\lambda} = \det(e_{\lambda'_{i}-i+j})$ shows that for such $\lambda$ one has
$s_{\lambda} = e_{r}s_{\mu}$ with $\mu = (\lambda_{1}-1,\dots,\lambda_{r}-1) \in \Pi_{r,n-r-1}$.
These $s_{\mu}$ form a basis of the ring on the left of the second displayed line.
\end{proof}

\begin{thm}[\protect{\cite[Theorem 11.2]{Panin:2010fk}}]
\label{T:Grass}
Let $(\shf U_{r,n},\phi_{r,n})$ be the tautological symplectic bundle of rank $2r$ on $\HGr(r,n)$.
Then for any
bigraded $\epsilon$-commutative ring cohomology theory $(A^{*,*},\partial,\times,1_{A})$ with
a symplectic Thom structure and any $X$ in $\Sm/S$
the map
\begin{equation}
\label{E:HGr.cohom.1}
A^{*,*}(X)[e_{1},\dots,e_{r}] / (h_{n-r+1},\dots,h_{n}) \xrightarrow{\cong} A^{*,*}(\HGr(r,n) \times X)
\end{equation}
sending $e_{i} \mapsto b_{i}(\shf U_{r,n},\phi_{r,n})$ for all $i$
is an isomorphism of bigraded rings.
\end{thm}

\begin{thm}[\protect{\cite[Theorem 11.4]{Panin:2010fk}}]
\label{T:HGr.lim.cohom.1}
Let $\alpha_{r,n} \colon \HGr(r,n) \hra \HGr(r,n+1)$
be the usual inclusion.
For any bigraded $\epsilon$-commutative ring cohomology theory $(A^{*,*},\partial,\times,1_{A})$ with
a symplectic Thom structure and any $X$ in $\Sm/S$
the map
\begin{equation*}
(\alpha_{r,n} \times 1)^{*} \colon A^{*,*}(\HGr(r,n+1) \times X) \to A^{*,*}(\HGr(r,n) \times X)
\end{equation*}
is a surjection which
the isomorphisms \eqref {E:HGr.cohom.1} identify
with the natural surjection
\[
A^{*,*}(X)[e_{1},\dots,e_{r}] / (h_{n-r+2},\dots,h_{n},h_{n+1})
\to
A^{*,*}(X)[e_{1},\dots,e_{r}] / (h_{n-r+1},h_{n-r+2},\dots,h_{n}).
\]
\end{thm}

\begin{thm}[\protect{\cite[Theorem 11.4]{Panin:2010fk}}]
\label{T:HGr.lim.cohom.2}
Let
$\beta_{r,n} \colon \HGr(r,n)  \to \HGr(1+r,1+n)$
be the usual inclusion.
For any bigraded $\epsilon$-commutative ring cohomology theory $(A^{*,*},\partial,\times,1_{A})$ with
a symplectic Thom structure and any $X$ in $\Sm/S$
the map
\begin{equation*}
(\beta_{r,n} \times 1)^{*} \colon A^{*,*}(\HGr(1+r,1+n) \times X) \to A^{*,*}(\HGr(r,n) \times X)
\end{equation*}
is a surjection which
the isomorphisms \eqref {E:HGr.cohom.1} identify
with the surjection
\[
A^{*,*}(\pt)[e_{1},\dots,e_{r},e_{r+1}] / (h_{n-r+1},\dots,h_{n},h_{n+1})
\to
A^{*,*}(\pt)[e_{1},\dots,e_{r}] / (h_{n-r+1},\dots,h_{n})
\]
of $A^{*,*}(\pt)[e_{1},\dots,e_{r}]$-algebras sending $e_{r+1} \mapsto 0$.
\end{thm}

\section{The cohomology of
{$BSp_{2r}$ and $\MSp_{2r}$}
{BSp\_2r and MSp\_2r}
}
\label{S:cohom.BSp.MSp}

\begin{thm}
\label{T:cohom.BSp}
Let $(A,\mu,e)$ be a commutative $T$-ring spectrum with a symplectic Thom structure on $A^{*,*}$.
Then the isomorphism $BSp_{2r} = \colim_{n} HGr(r,n)$ induces isomorphisms
\begin{equation*}
A^{*,*}(BSp_{2r})
\xrightarrow{\cong}
\varprojlim\limits_{n \to \infty} A^{*,*}(\HGr(r,n) )
\xleftarrow{\cong}
A^{*,*}(\pt) [[b_{1},\dots,b_{r}]]^{\homog}
\end{equation*}
of bigraded rings.
The second isomorphism
sends the variable $b_{i}$ to the inverse system of $i^{\text{th}}$ Borel classes
$(b_{i}(\shf U_{r,n}))_{n \geq r}$.
\end{thm}


Here $A^{*,*}(\pt) [[b_{1},\dots,b_{r}]]^{\homog}$ is the bigraded ring of homogeneous power series.
The $\varprojlim$ is taken in the category of bigraded rings.

\begin{proof}
We have $BSp_{2r} = \colim_{n} HGr(r,n)$, so by Theorem \ref{T:hocolim} we have an exact sequence
\[
0 \to \sideset{}{^{1}}\varprojlim_{n\in\NN} A^{*-1,*}(HGr(r,n))
\to A^{*,*}(BSp_{2r}) \to
\varprojlim_{n\in\NN} A^{*,*}(HGr(r,n)) \to 0.
\]
The connecting maps are the $\alpha_{r,n}^{*}$ of Theorem \ref {T:HGr.lim.cohom.1}, which are surjective.
So the $\varprojlim^{1}$ vanishes, and the first map of the statement of the theorem is an isomorphism.

Let $I_{d} \subset A^{*,*}(\pt)[b_{1},\dots,b_{r}]$ be the two-sided ideal generated by the monomials
$b_{1}^{a_{1}}p_{2}^{a_{2}} \dots b_{r}^{a_{r}}$ with $\sum i a_{i} \geq d$.  We have inclusions
\(
I_{r(n-r)+1} \subset (h_{n-r+1},\dots,h_{n}) \subset I_{n-r+1}
\)
because all partitions $\lambda \in \Pi_{r,n-r}$ have
$\lvert \lambda \rvert \leq r(n-r)$.
So we have
\begin{multline*}
\varprojlim_{n} A^{*,*}(HGr(r,n)) \cong
\varprojlim_{n} A^{*,*}(\pt)[b_{1},\dots,b_{r}]/(h_{n-r+1},\dots,h_{n})
\\ \cong \varprojlim_{d} A^{*,*}(\pt)[b_{1},\dots,b_{r}]/I_{d}
= A^{*,*}(\pt) [[b_{1},\dots,b_{r}]]^{\homog}.
\tag*{\qedhere}
\end{multline*}
\end{proof}

\begin{thm}
\label{T:cohom.MSp.2r}
Let $(A,\mu,e)$ be a commutative $T$-ring spectrum with a symplectic Thom structure on $A^{*,*}$.
Then for any $r$ and $n$ the motivic homotopy equivalence
\[
\Th \shf U_{HGr(r,n)} \cong HGr(r,1+n)/HGr(r-1,n)
\]
of Theorem \ref{T:BSp/BSp} induces isomorphisms
\begin{equation}
\label{E:cohom.Thom}
A^{*,*}(\pt)[b_{1},\dots,b_{r}]/(h_{n-r+1},\dots,h_{n})
\xra[\cong]{\cup b_{r}}
A^{*+4r,*+2r}(\Th \shf U_{HGr(r,n)})
\end{equation}
of two-sided bigraded modules over
$A^{*,*}(HGr(r,n)) \cong A^{*,*}(\pt)[b_{1},\dots,b_{r}]/(h_{n-r+1},\dots,h_{n})$.
Moreover, the isomorphism $\MSp_{2r} = \colim_{n} \Th \shf U_{HGr(r,n)}$ induces
isomorphisms of bigraded modules over
$A^{*,*}(BSp_{2r}) \cong A^{*,*}(\pt)[[b_{1},\dots,b_{r}]]^{\homog}$
\begin{equation}
\label{E:A(MSp.2r)}
A^{*+4r,*+2r}(\MSp_{2r})
\xra{\cong}
\varprojlim\limits_{n \to \infty} A^{*+4r,*+2r}(\Th \shf U_{HGr(r,n)} )
\xla{\cong}
b_{r} A^{*,*}(\pt) [[b_{1},\dots,b_{r}]]^{\homog}
\end{equation}
\end{thm}

\begin{proof}
Applying $A^{*,*}$ to the cofiber sequence
\[
HGr(r-1,n) \xra{\beta_{r-1,n}} HGr(r,1+n) \lra \Th \shf U_{HGr(r,n)}
\]
gives a long exact sequence of cohomology.  The map
$\beta_{r-1,n}$ induces a surjection of cohomology groups by Theorem \ref{T:HGr.lim.cohom.2}, so we
have $A^{*,*}(\Th \shf U_{HGr(r,n)}) \cong \ker \beta_{r-1,n}^{*}$.
This kernel is identified by Proposition \ref{P:sym.func.2}, giving the isomorphism \eqref{E:cohom.Thom}.
In principle this is an isomorphism of two-sided modules over
$A^{*,*}(HGr(r,1+n)) \cong A^{*,*}(\pt)[b_{1},\dots,b_{r}]/(h_{n-r+2},\dots,h_{n+1})$.  But the modules
are annihilated by $h_{n-r+1}$, so they are also two-sided modules over the quotient ring
$A^{*,*}(HGr(r,n))$.

The inclusion $\Th \shf U_{HGr(r,n)} \hra \Th \shf U_{HGr(r,n+1)}$ induces a commutative
diagram
\begin{equation*}
\vcenter{\xymatrix @M=5pt @C=35pt {
A^{*,*}(\pt)[b_{1},\dots,b_{r}]/(h_{n-r+2},\dots,h_{n},h_{n+1})
\ar[r]_-{\cong}^-{\cup b_{r}} \ar@{->>}[d]
& A^{*+4r,*+2r}(\Th \shf U_{HGr(r,n+1)}) \ar@{->>}[d]
\\
A^{*,*}(\pt)[b_{1},\dots,b_{r}]/(h_{n-r+1},h_{n-r+1}\dots,h_{n})
\ar[r]_-{\cong}^-{\cup b_{r}}
& A^{*+4r,*+2r}(\Th \shf U_{HGr(r,n)}).
}}
\end{equation*}
The inverse limit gives the righthand isomorphism of \eqref {E:A(MSp.2r)}.
The vertical map on the left is surjective, so the vertical map on the right is as well.
Therefore the $\varprojlim^{1}$ vanishes in the exact sequence
\[
0 \to \sideset{}{^{1}}\varprojlim_{n\in\NN} A^{*-1,*}(\Th \shf U_{HGr(r,n)})
\to A^{*,*}(\MSp_{2r}) \to
\varprojlim_{n\in\NN} A^{*,*}(\Th \shf U_{HGr(r,n)}) \to 0.
\]
obtained from Theorem \ref{T:hocolim}.
\end{proof}

For any $r$ let $z_{2r} \colon BSp_{2r} \to \Th \shf U_{BSp_{2r}} = \MSp_{2r}$
be the structure map induced by
the zero section of the tautological symplectic bundle $\shf U_{BSp_{2r}} \to BSp_{2r}$.

\begin{thm}
\label{T:cohom.MSp.2r.ideal}
Let $(A,\mu,e)$ be a commutative $T$-ring spectrum with a symplectic Thom structure on $A^{*,*}$.
The map $z_{2r}^{*} \colon A^{*,*}(\MSp_{2r}) \to A^{*,*}(BSp_{2r})$ of two-sided
$A^{*,*}(BSp_{2r})$-modules is injective
and identifies $A^{*,*}(\MSp_{2r})$ with the two-sided principal ideal generated by
$b_{r} \in A^{4r,2r}(BSp_{2r})$.
\end{thm}

\begin{proof}
By Theorem \ref{T:MSp=BSp/BSp} we have cofiber sequence
$BSp_{2r-2} \xra{i_{2r}} BSp_{2r} \xra{z_{2r}} \MSp_{2r}$ yielding a long exact sequence of cohomology groups.
By the previous theorems these are isomorphic to (in simplified notation)
\begin{equation}
\label{E:long.exact}
\cdots \to A^{*-4r,*-2r}[[b_{1},\dots,b_{r}]] \xra{z_{2r}^{*}}
A^{*,*}[[b_{1},\dots,b_{r}]] \xra{i_{2r}^{*}}
A^{*,*}[[b_{1},\dots,b_{r-1}]] \to \cdots
\end{equation}
Since $i_{2r}$ is the colimit of the inclusion maps $\beta_{r-1,n} \colon A(r-1,n) \to A(r,1+n)$
of Theorem \ref{T:HGr.lim.cohom.2}, $i_{2r}^{*}$ is the quotient by the ideal generated by $b_{r}$.
It is surjective in all bidegrees.
It follows that $z_{2r}^{*}$ is injective in all bidegees and is the inclusion of that ideal.
\end{proof}

The direct sum of symplectic bundle induces compatible monoid structures on the
$BSp_{2r}$ and the $\MSp_{2r}$.  So the following diagram commutes for all $r$ and $s$
\begin{equation}
\label{E:mult.BSp.MSp}
\vcenter{\xymatrix @M=5pt @C=35pt {
BSp_{2r} \times BSp_{2s} \ar[r]^-{m_{rs}} \ar[d]_-{z_{2r} \times z_{2s}}
& BSp_{2r+2s} \ar[d]^-{z_{2r+2s}}
\\
\MSp_{2r} \wedge \MSp_{2s} \ar[r]^-{\mu_{rs}}
& \MSp_{2r+2s}.
}}
\end{equation}

\begin{thm}
\label{T:cohom.BSp.BSp}
Let $(A,\mu,e)$ be a commutative $T$-ring spectrum with a symplectic Thom structure on $A^{*,*}$.
Then the isomorphisms
\begin{align*}
BSp_{2r} \times BSp_{2s} & = \colim\nolimits_{n} \bigl( HGr(r,rn) \times HGr(s,sn) \bigr)
\\
\MSp_{2r} \wedge \MSp_{2s} & = \colim\nolimits_{n}
(\Th \shf U_{HGr(r,rn)} \wedge \Th \shf U_{HGr(s,sn)})
\end{align*}
induces a commutative diagram of isomorphisms and monomorphisms of two-sided graded
$A^{*,*}(BSp_{2r} \times BSp_{2s})$-modules
\begin{small}
\begin{equation*}
\xymatrix @M=5pt @C=10pt {
A^{*,*}(\MSp_{2r} \wedge \MSp_{2s})
\ar[r]^-{\cong} \ar@{_{(}->}[d]_-{(z_{2r} \times z_{2s})^{*}}
&
\raisebox{-15pt}{$\varprojlim\limits_{n \to \infty} A^{*,*}(\Th \shf U_{r,rn} \wedge \Th \shf U_{s,sn})$}
\ar@{_{(}->}[d]
& \ar[l]_-{\cong}
p'_{r}p''_{s}A^{*,*} [[p'_{1},\dots,p'_{r},p''_{1},\dots,p''_{s}]]
\ar@{_{(}->}[d]^-{\text{\textup{inclusion}}}
\\
A^{*,*}(BSp_{2r} \times BSp_{2s})
\ar[r]^-{\cong}
&
\raisebox{-15pt}{$\varprojlim\limits_{n \to \infty} A^{*,*}(\HGr(r,rn) \times  \HGr(s,sn))$}
& \ar[l]_-{\cong}
A^{*,*} [[p'_{1},\dots,p'_{r},p''_{1},\dots,p''_{s}]].
}
\end{equation*}
\end{small}%
Moreover, these isomorphisms identify the diagram obtained by applying $A^{*,*}$ to
\eqref{E:mult.BSp.MSp} with the diagram of rings and ideals
\[
\xymatrix @M=5pt @C=30pt {
p_{r+s}A^{*,*}(\pt)[[b_{1},\dots,p_{r+s}]]^{\homog} \ar[r] \ar@{_{(}->}[d]
&
p'_{r}p''_{s} A^{*,*}(\pt) [[p'_{1},\dots,p'_{r},p''_{1},\dots,p''_{s}]]^{\homog} \ar@{_{(}->}[d]
\\
A^{*,*}(\pt)[[b_{1},\dots,p_{r+s}]]^{\homog} \ar[r]
&
A^{*,*}(\pt) [[p'_{1},\dots,p'_{r},p''_{1},\dots,p''_{s}]]^{\homog}
}
\]
where the horizontal maps send $b_{i} \mapsto p'_{i} + \sum_{j=1}^{i-1} p'_{i-j}p''_{j}+ p''_{i}$.
Moreover, the horizontal maps of the last diagram are also injective.
\end{thm}

\begin{proof}
The construction of the first diagram is much the same as in the previous theorems.
The second diagram follows.  For the last statement of the theorem, let
$t_{1},\dots,t_{r+s}$ be independent indeterminates of bidegree $(2,1)$.  Then the composition of
the bottom horizontal map with the map
\[
A^{*,*}(\pt) [[p'_{1},\dots,p'_{r},p''_{1},\dots,p''_{s}]]^{\homog}
\to
A^{*,*}(\pt) [[t_{1},\dots,t_{r+s}]]^{\homog}
\]
sending $p'_{i} \mapsto e_{i}(t_{1},\dots,t_{r})$ and $p''_{j} \mapsto e_{j}(t_{r+1},\dots,t_{r+s})$
is the inclusion of the ring of symmetric homogeneous power series in the ring of homogeneous
power series.  That is injective.
\end{proof}

The final calculation in this section ought to be that of $A^{*,*}(\MSp)$.  However, we will
put this off until Theorem \ref{T:cohom.MSp} because we wish to make the calculation using
a symplectic Thom classes theory and not just a symplectic Thom structure.

\section{Tautological Thom elements}

Suppose that $(A,\mu,e)$ is a commutative $T$-ring spectrum.
Let $\vartheta \in A^{4,2}(\MSp_{2})$.

We associate to $\vartheta$ and a symplectic bundle $(E,\phi)$ of rank $2$ over an
$X$ in $\Sm/S$ a class $\thom^{\vartheta}(E,\phi)$ defined as follows.  By assumption
the scheme $X$ admits an ample family of line bundles.  So there exists an
affine bundle $f \colon Y \to X$ with $Y$ an affine scheme.  Then for some $p$ there exist global sections
$s_{1},\dots,s_{p}$ of $f^{*}E^{\vee}$ generating $f^{*}E^{\vee}$.  There then exist global functions
$a_{ij}$ on $Y$ such that $f^{*} \phi = \sum_{1\leq i < j \leq p} a_{ij} s_{i} \wedge s_{j}$.
We set $t_{i} = \sum_{j=i+1}^{p} a_{ij}s_{j}$ so that we have
$\sum_{i} s_{i} \wedge t_{i} = f^{*}\phi$.  The map
$(s_{1},t_{1},\dots, s_{p},t_{p}) \colon f^{*}E \to \OO_{Y}^{\oplus 2p}$
embeds $(f^{*}E, f^{*}\phi)$ as a symplectic subbundle of $(\OO_{Y}^{\oplus 2p}, \omega_{2p})$.
So it is classified by a map $\psi \colon Y \to HGr(1,p) = HP^{p-1}$ such that
$\psi^{*}(\TSp{1,p},\phi_{1,p}) = f^{*}(E,\phi)$.
This gives us maps of (ind)-schemes
\begin{equation}
\label{E:X.to.BSp}
X \xla[\text{$\Aff^{N}$-bundle}]{f} Y
\xra{\psi}
HP^{p-1}
\xra{\text{inclusion}} BSp_{2} = HP^{\infty}
\end{equation}
and of pointed motivic spaces
\begin{equation}
\label{E:ThE.to.A.1}
\Th E \xla[\sim\text{mot}]{\overline f} \Th f^{*}E
\, \cong \, \Th \widetilde\psi^{*} \TSp{1,p}
\xra{\overline{\psi}}
\Th \TSp{1,p}
\end{equation}
of pointed motivic spaces, which can be composed with the maps
\begin{equation}
\label{E:ThE.to.A.2}
\Th \TSp{1,p} \xra{\text{inclusion}}
 \MSp_{2} \xra{\vartheta} A \wedge T^{\wedge 2}.
\end{equation}
in $SH(S)$.
The composition of \eqref{E:ThE.to.A.1} and \eqref{E:ThE.to.A.2} gives a class
\[
\thom^{\vartheta}(E,\phi) \in Hom_{SH(S)}(\Th E, A \wedge T^{\wedge 2}) = A^{4,2}(E,E-X).
\]
The following lemma is proven in the same way as Lemma \ref{L:well.defined.SL}.

\begin{lem}
\label{L:well.defined.Sp.1}
The classes $\thom^{\vartheta}(E,\phi)$ depend only on the rank $2$ symplectic bundle $(E,\phi)$
and the morphism  $\vartheta \colon \Sigma_{T}^{\infty}\MSp_{2}(-2) \to A$ in $SH(S)$.
\end{lem}

Recall the inclusion $e_{2}^{Sp} \colon T^{\wedge 2} \to \MSp_{2}$ of \eqref{E:e.Sp}.

\begin{thm}
\label{T:a.alpha}
Let $(A,\mu,e)$ be a commutative $T$-ring spectrum.
Then the map which assigns to a class $\vartheta$ as above the family of classes
$\thom^{\vartheta}(E,\phi)$ is a bijection between the sets of

\parens{$\alpha$} classes $\vartheta \in A^{4,2}(\MSp_{2})$ with
$\vartheta |_{T^{\wedge 2}} = \Sigma_{T}^{2}1_{A}$ in $A^{4,2}(T^{\wedge 2})$, and

\parens{a} symplectic Thom structures on the bigraded $\epsilon$-commutative ring cohomology theory
$(A^{*,*},\partial,\times,1_{A})$ such that for the trivial rank $2$
bundle $\Aff^{2} \to \pt$ we have
$\thom(\Aff^{2},\omega_{2}) = \Sigma_{T}^{2}1_{A}$ in $A^{4,2}(T^{\wedge 2})$.
\end{thm}

\begin{proof}
A proof similar to that of Theorem \ref{T:thom.phi} shows that for a $\vartheta$ as in ($\alpha$),
the family of classes $\thom^{\vartheta}(E,\phi)$ form a symplectic Thom structure with the stated
normalization condition.  Note that this uses the fact that all symplectic bundles are locally
trivial in the Zariski topology.

Now suppose we have a symplectic Thom structure with the stated normalization condition.
For every $n$ the tautological rank $2$ symplectic bundle bundle over $HP^{n-1}$ has a Thom class
which we will abbreviate as
\[
\thom_{n} =
\thom(\shf U_{HP^{n-1}},\phi_{HP^{n-1}}) \in A^{4,2}(\Th \shf U_{HP^{n-1}}).
\]
Pullback along the inclusion $\Th \shf U_{HP^{n-1}} \to \Th \shf U_{HP^{n}}$ sends
$\thom_{n+1} \mapsto \thom_{n}$.
So as $n$ varies, we get an element
\[
\bar\vartheta = (\thom_{n} )_{n\in\NN}
\in \varprojlim A^{4,2}(\Th \shf U_{HP^{n-1}}).
\]
We have $\MSp_{2} = \colim \Th \shf U_{HP^{n-1}}$, and by Theorem \ref{T:cohom.MSp.2r}
the natural map
\begin{equation*}
A^{*,*}(\MSp_{2}) \xra{\cong} \varprojlim A^{*,*}(\Th \shf U_{HP^{n-1}}).
\end{equation*}
is an isomorphism.
Let $\vartheta \in A^{4,2}(\MSp_{2})$ be the unique class lifting $\bar\vartheta$.  As in the proof of
Theorem \ref{T:univ.SL} we have $\thom^{\vartheta}(E,\phi) = \thom(E,\phi)$ for all rank $2$
symplectic bundles.  Moreover, for $\vartheta$ and $\xi$ in $A^{4,2}(\MSp_{2})$ we have
$\thom^{\vartheta}(E,\phi) = \thom^{\xi}(E,\phi)$ for all symplectic bundles if and only if
$\vartheta$ and $\xi$ have the same image in the inverse limit.  But that happens only for
$\vartheta = \xi$.
\end{proof}

\begin{defn}
The class $\vartheta \in A^{4,2}(\MSp_{2})$ is the \emph{tautological Thom element}
of the symplectic orientation on $A^{*,*}$ whose
rank $2$ symplectic Thom classes are the $\thom^{\vartheta}(E,\phi)$.
\end{defn}

The canonical morphism $u_{2} \colon \Sigma_{T}^{\infty} \MSp_{2}(-2) \to \MSp$ which is part of the
counit of the adjunction between $\Sigma_{T}^{\infty}(-2)$ and its right adjoint the forgetful
functor $\Ev_{2}$ defines an element $\vartheta_{\MSp} \in \MSp^{4,2}(\MSp_{2})$.  It satisfies
$\vartheta_{\MSp} |_{T^{\wedge 2}} = \Sigma_{T}^{2}1_{\MSp}$ because both elements correspond to the
composition $u_{2} \circ e_{2}^{Sp} \colon T^{\wedge 2} \to \MSp \wedge T^{\wedge 2}$.

\begin{defn}
The \emph{standard symplectic Thom structure} on $\MSp^{*,*}$ is the one whose universal
Thom element is the $\vartheta_{\MSp}$ we have just described.
\end{defn}

\section{Tautological Borel elements}

The bigraded version of the definition of a Borel structure
\cite[Definition 12.1]{Panin:2010fk} is as follows.

\begin{defn}
\label{D:Pont}
A \emph{Borel structure} on a bigraded $\epsilon$-commutative
ring cohomology theory $(A^{*,*},\partial,\times,1_{A})$ on $\SmOp/S$ is a
rule which assigns to each rank $2$ symplectic bundle $(E,\phi)$ over an $X$
in $\Sm/S$ an element $b_{1}(E,\phi) \in A^{4,2}(X)$ with the following properties:

\begin{enumerate}
\item
For $(E_{1},\phi_{1}) \cong (E_{2},\phi_{2})$ we have $b_{1}(E_{1},\phi_{1}) = b_{1}(E_{2},\phi_{2})$.

\item
For a morphism $f \colon Y \to X$ we have $f^{*}b_{1}(E,\phi)) = b_{1}(f^{*}E,f^{*}\phi)$.

\item
For the tautological rank $2$ symplectic subbundle
$(\shf U_{HP^{1}},\phi_{\HP^{1}})$ on $\HP^{1}$ the map
\[
(1,b_{1}(\shf U_{HP^{1}},\phi_{\HP^{1}})) \colon A^{*,*}(X) \oplus A^{*-4,*-2}(X) \to
A^{*,*}(\HP^{1} \times X)
\]
is an isomorphism for all $X$ in $\Sm/S$.

\item
For the trivial rank $2$ symplectic bundle $(\Aff^{2},\omega_{2})$ over $\pt$
we have $b_{1}(\Aff^{2},\omega_{2}) = 0$ in $A^{4,2}(\pt)$.

\end{enumerate}
\end{defn}

The Borel classes associated to a symplectic Thom structure by the formula
$b_{1}(E,\phi) = -z^{*} \thom(E,\phi) \in A^{4,2}(X)$ of Definition \ref{D:sp.thom}
form a Borel structure because of the functoriality of the Thom classes, the
quaternionic projective bundle theorem \ref{T:H.proj.bdl.1} and
Corollary \ref{C:pont.triv}.

For $r=1$ the diagram \eqref{E:BSp/BSp} of morphisms in $H_{\bullet}(S)$ becomes
\begin{equation}
\label{E:can.htpy.equiv}
\vcenter{
\xymatrix @M=5pt @C=40pt {
T^{\wedge 2} \ar[r]^-{e_{1}^{Sp}} \ar[d]_-{\cong}
& \MSp_{2} \ar[d]^-{\cong}
\\
(HP^{1},h_{\infty}) \ar[r]_-{\text{\textup{inclusion}}}
& (HP^{\infty},h_{\infty})
}}
\end{equation}
with $h_{\infty} = \pt \to HP^{1}$ a point such that $h_{\infty}^{*}(\shf U_{HP^{1}},\phi_{HP^{1}})$
is a trivial symplectic bundle.
We will call the two vertical arrows the \emph{canonical motivic homotopy equivalences}.

Suppose that $(A,\mu,e)$ is a commutative $T$-ring spectrum.
Let $\varrho \in A^{4,2}(HP^{\infty},h_{\infty})$.
For a rank $2$ symplectic bundle $(E,\phi)$ over $X$ the construction
of \eqref{E:X.to.BSp} composed with the quotient by the pointing and with $\varrho$
gives us a zigzag
\begin{equation}
X \xla[\sim]{\text{$\Aff^{N}$-bundle}} Y \to
HP^{\infty} \to (HP^{\infty},h_{\infty})  \xra{\varrho} A \wedge T^{\wedge 2}
\end{equation}
in which the pullbacks to $Y$ of $(E,\phi)$ and of $(\shf U_{HP^{\infty}}, \phi_{HP^{\infty}})$
are isomorphic.  The composition is a class $b_{1}^{\varrho}(E,\phi) \in A^{4,2}(X)$.
This class depends only on $(E,\phi)$ and $\varrho$ by arguments similar to those of
Lemmas \ref{L:well.defined.SL} and \ref{L:well.defined.Sp.1}.

\begin{thm}
\label{T:b.beta}
Let $(A,\mu,e)$ be a commutative $T$-ring spectrum.
Then the map which assigns to a class $\varrho$ as above the family of classes
$b_{1}^{\varrho}(E,\phi)$ is a bijection between the sets of

\parens{$\beta$} classes $\varrho \in A^{4,2}(HP^{\infty},h_{\infty})$ with
$\varrho|_{HP^{1}} \in A^{4,2}(HP^{1},h_{\infty})$ corresponding to
$-\Sigma_{T}^{2}1_{A} \in A^{4,2}(T^{\wedge 2})$ under the canonical motivic homotopy equivalence
$(HP^{1},h_{\infty}) \cong T^{\wedge 2}$, and

\parens{b} Borel structures on
$(A^{*,*},\partial,\times,1_{A})$ for which
$b_{1}(\shf U_{HP^{1}},\phi_{HP^{1}}) \in A^{4,2}(HP^{1},h_{\infty}) \subset A^{4,2}(HP^{1})$
corresponds to $-\Sigma_{T}^{2}1_{A}$ in $A^{4,2}(T^{\wedge 2})$ under the canonical
motivic homotopy equivalence $(HP^{1},h_{\infty}) \simeq T^{\wedge 2}$.
\end{thm}

The proof is like that of Theorem \ref{T:a.alpha}.  The classes $b_{1}^{\varrho}(E,\phi)$ satisfy
condition (3) of Definition \ref{D:Pont} because of an argument like Lemma \ref{L:triv.Thom}
and the isomorphism
$T^{\wedge 2} \cong (HP^{1},h_{\infty})$.  They satisfy condition (4) because
$b_{1}^{\varrho}(\Aff^{2},\omega_{2}) = h_{\infty}^{*}\varrho = 0$.
The proof that there is a unique $\varrho$ corresponding to each Borel structure invokes
the isomorphism $A^{*,*}(HP^{\infty}) \cong \varprojlim A^{*,*}(HP^{n})$ which is the case
$r=1$ of Theorem \ref{T:cohom.BSp}.

\begin{defn}
The class $\varrho \in A^{4,2}(HP^{\infty},h_{\infty})$ is the \emph{tautological Borel element}
of the symplectic orientation on $A^{*,*}$ whose
rank $2$ Borel classes are the $b_{1}^{\varrho}(E,\phi)$.
\end{defn}

\begin{thm}
\label{T:alpha.beta}
Let $(A,\mu,e)$ be a commutative $T$-ring spectrum.
Then the canonical motivic homotopy equivalence
$\MSp_{2} \cong (HP^{\infty},h_{\infty})$ plus change-of-sign
gives a bijection between the sets of

\parens{$\alpha$} the tautological Thom elements $\vartheta$ of
Theorem \ref{T:a.alpha} and

\parens{$\beta$} the tautological Borel elements $\varrho$ of Theorem \ref{T:b.beta}.

The composition
\parens{a} $\leftrightarrow$ \parens{$\alpha$} $\leftrightarrow$ \parens{$\beta$} $\leftrightarrow$ \parens{b} with the bijections of Theorems \ref{T:a.alpha} and \ref{T:b.beta}
is the same as the rule which assigns to a symplectic Thom structure with classes $\thom(E,\phi)$
the Borel structure with classes $b_{1}(E,\phi) = -z^{*}\thom(E,\phi)$ for
$z \colon X \to \Th E$ the structural map of the Thom space.
\end{thm}

\begin{proof}
The first statement follows from the existence and compatibility of the canonical motivic homotopy
equivalences of \eqref{E:can.htpy.equiv}.  For the second, given a rank $2$ symplectic bundle
$(E,\phi)$ on $X$ we have a diagram
\[
\xymatrix @M=5pt {
X \ar[d]_-{z}
& Y \ar[l]_-{f}^-{\sim} \ar[r] \ar[d]
& HP^{\infty} \ar[r] \ar[d]
&  (HP^{\infty},h_{\infty})  \ar[d]^-{\varrho} \ar[ld]_-{\cong}
\\
\Th E
&  \Th f^{*}E \ar[l]_-{\sim} \ar[r]
& \MSp_{2} \ar[r]^-{-\vartheta}
& A \wedge T^{\wedge 2}
}
\]
in which the squares commute by compatibility of the structural maps of Thom spaces with
pullbacks, the upper triangle commutes by Theorem \ref {T:MSp=BSp/BSp}, and the lower triangle
commutes because of the rule giving the bijection ($\alpha$) $\leftrightarrow$ ($\beta$).
We deduce the equality $b_{1}^{\varrho}(E,\phi) = -z^{*}\thom^{\vartheta}(E,\phi)$
in $Hom_{SH(S)}(X, A \wedge T^{\wedge 2}) = A^{4,2}(X)$.
\end{proof}

The bigraded version of the definition of a Borel classes theory
\cite[Definition 14.1]{Panin:2010fk} is as follows.

\begin{defn}
\label{D:Pont.classes}
A \emph{Borel classes theory} on a bigraded $\epsilon$-commutative
ring cohomology theory $(A^{*,*},\partial,\times,1_{A})$ on $\SmOp/S$
is a rule assigning to every
symplectic bundle $(F,\psi)$ over every $X$ in $\Sm/S$
elements $b_{i}(F,\psi) \in A^{4i,2i}(X)$ for all
$i \geq 1$ satisfying
\begin{enumerate}
\item
For $(F_{1},\psi_{1}) \cong (F_{2},\psi_{2})$ we have $b_{i}(F_{1},\psi_{1}) = b_{i}(F_{2},\psi_{2})$
for all $i$.

\item
For a morphism $f \colon Y \to S$ we have $f^{*}b_{i}(F,\psi) = b_{i}(f^{*}E,f^{*}\psi)$ for all $i$.

\item
For the tautological rank $2$ symplectic subbundle
$(\shf U_{HP^{1}},\phi_{HP^{1}})$ on $\HP^{1}$ the maps
\[
(1,b_{1}(\shf U_{HP^{1}},\phi_{HP^{1}})) \colon A^{*,*}(X) \oplus A^{*-4,*-2}(X) \to
A^{*,*}(\HP^{1} \times X)
\]
are isomorphisms for all $X$.

\item
For the trivial rank $2$ symplectic bundle $(\Aff^{2},\omega_{2})$ over $\pt$
we have $b_{1}(\Aff^{2},\omega_{2}) = 0$ in $A^{4,2}(\pt)$.

\item
For an orthogonal direct sum of symplectic bundles $(F,\psi) \cong (F_{1},\psi_{1}) \oplus (F_{2},\psi_{2})$
we have $b_{i}(F,\psi) = b_{i}(F_{1},\psi_{1}) + \sum_{j=1}^{i-1} b_{i-j}(F_{1},\psi_{1}) b_{j}(F_{2},\psi_{2})
+ b_{i}(F_{2},\psi_{2})$ for all $i$.

\item
For $(F,\psi)$ of rank $2r$ we have $b_{i}(F,\psi) = 0$ for $i > r$.
\end{enumerate}
\end{defn}

One may also set $p_{0}(F,\psi) = 1$ and even $b_{i}(F,\psi) = 0$ for $i < 0$.

Definition \ref{D:PontryaginClasses} associates Borel classes to a symplectic Thom structure on
$(A^{*,*},\partial,\times, 1_{A})$.  They form a Borel classes theory because the
quaternionic projective bundle Theorems \ref{T:H.proj.bdl.1} and \ref{T:H.proj.bdl.2},
Corollary \ref{C:pont.triv} and the Cartan sum formula (Theorem \ref{T:Cartan}).

\begin{thm}
\label{T:b.c}
Let $(A,\mu,e)$ be a commutative $T$-ring spectrum.
Then the forgetful map
gives a bijection between the sets of

\parens{c} Borel classes theories on $(A^{*,*},\partial,\times, 1_{A})$
with the normalization condition on $b_{1}(\shf U_{HP^{1}},\phi_{HP^{1}})$ of Theorem \ref{T:b.beta} and

\parens{b} Borel structures on $(A^{*,*},\partial,\times, 1_{A})$
with the same normalization condition.

The inverse bijection is given by assigning to a Borel structure first the symplectic
Thom structure associated to it by Theorem \ref{T:alpha.beta} and then the
the Borel classes theory associated to the symplectic Thom structure by
Definition \ref{D:PontryaginClasses}.
\end{thm}

\begin{proof}
The chain of associations (\textit{b}) $\to$ (\textit{a}) $\to$ (\textit{c}) $\to$ (\textit{b})
gives the identity because for rank $2$ symplectic bundles the classes $b_{1}(E,\phi)$ given in
Definitions \ref{D:sp.thom} and \ref{D:PontryaginClasses} coincide.

The chain of associations (\textit{c}) $\to$ (\textit{b}) $\to$ (\textit{a}) $\to$ (\textit{c})
gives the identity because for a symplectic bundle $(F,\psi)$ of rank $2r$ on $X$ if we let
$\pi \colon HP(F,\psi) \to X$ be the associated quaternionic projective bundle with rank $2$
tautological subbundle $(\shf U,\phi)$, then from the orthogonal direct sum
$\pi^{*}(F,\psi) = (\shf U,\phi) \oplus (\shf U,\phi)^{\perp}$ and the axioms we get
\[
0 = (-1)^{r} b_{r}\bigl( (\shf U,\phi)^{\perp} \bigr) = b_{1}(\shf U,\phi)^{r} -
\pi^{*}b_{1}(F,\psi) \cup b_{1}(\shf U,\phi)^{r-1} + \cdots + (-1)^{r} \pi^{*} b_{r}(F,\psi).
\]
Hence the Borel classes defined by (\textit{c}) $\to$ (\textit{b}) $\to$ (\textit{a}) $\to$ (\textit{c})
coincide with the original ones.
\end{proof}

\section{Higher rank symplectic Thom classes}

The bigraded version of the definition of a symplectic Thom classes theory
\cite[Definition 14.2]{Panin:2010fk} is as follows.

\begin{defn}
\label{D:thom.class}
A \emph{symplectic Thom classes theory}
on a bigraded $\epsilon$-commutative
ring cohomology theory $(A^{*,*},\partial,\times,1_{A})$ on $\SmOp/S$
is a rule assigning to every
symplectic bundle $(F,\psi)$ over every scheme $X$ in $\Sm/S$
an element $\thom(F,\psi) \in A^{4r,2r}(F,F-X)$ with $2r = \rk F$
with the following properties\textup{:}
\begin{enumerate}

\item For an isomorphism $u \colon (F,\psi) \cong (F_{1},\psi_{1})$
we have $\thom(F,\psi) = u^{*}\thom(F_{1},\psi_{1})$.

\item For $f \colon Y \to X$, writing $f_{F} \colon f^{*}F \to F$ for the pullback,
we have $f_{F}^{*}\thom(F,\psi) = \thom(f^{*}F,f^{*}\psi)) \in A^{4r,2r}(f^{*}F,f^{*}F-Y)$.

\item The maps $\cup \thom(F,\psi) \colon A^{*,*}(X) \to A^{*+4r,*+2r}(F,F-X)$ are isomorphisms.

\item We have $\thom \bigl( (F_{1},\psi_{1}) \oplus (F_{2},\psi_{2}) \bigr)
= q_{1}^{*}\thom(F_{1},\psi_{1}) \cup q_{2}^{*}\thom(F_{2},\psi_{2})$,
where $q_{1},q_{2}$ are the projections from $F_{1} \oplus F_{2}$
onto its factors.  Moreover, for the zero bundle $\boldsymbol 0 \to \pt$ we have
$\thom(\boldsymbol 0) = 1_{A} \in A^{0,0}(\pt)$.
\end{enumerate}

The classes $\thom(F,\psi)$ are \emph{symplectic Thom classes}.
\end{defn}

Let  $(A,\mu,e)$ be a commutative $T$-ring spectrum.
Suppose we have a sequence of classes
$\boldsymbol{\vartheta} = (\vartheta_{1},\vartheta_{2},\vartheta_{3},\dots)$
with  $\vartheta_{r} \in A^{4r,2r}(\MSp_{2r})$ for each $r$.
Then for any symplectic bundle $(F,\psi)$ of rank $2r$ over $X$ one can use $\vartheta_{r}$
to define a class $\thom^{\boldsymbol\vartheta}(F,\psi)$ by the construction already described
in \eqref {E:X.to.BSp}--\eqref{E:ThE.to.A.2} for rank $2$.  For a rank $0$ bundle
$\boldsymbol 0_{X} \to X$ we set $\thom(\boldsymbol{0}_{X}) = 1_{X} \in A^{0,0}(X)$.
These classes are well-defined by the same argument as in Lemma \ref{L:well.defined.SL} and
\ref{L:well.defined.Sp.1}.

Recall the inclusion $e_{2}^{Sp} \colon T^{\wedge 2} \to \MSp_{2}$ of \eqref{E:e.Sp} and
the monoid maps $\mu_{rs} \colon \MSp_{2r} \wedge \MSp_{2s} \to \MSp_{2r+2s}$
of \eqref{E:monoid.Sp}.

\begin{thm}
\label{T:d.delta}
Let $(A,\mu,e)$ be a commutative $T$-ring spectrum.
Then the map which assigns to a sequence of classes
$\boldsymbol{\vartheta} = (\vartheta_{1},\vartheta_{2},\vartheta_{3},\dots)$
as above the family of classes
$\thom^{\boldsymbol\vartheta}(F,\psi)$ is a bijection between the sets of

\parens{$\delta$} sequences of classes
$\boldsymbol{\vartheta} = (\vartheta_{1},\vartheta_{2},\vartheta_{3},\dots)$
with $\vartheta_{r} \in A^{4r,2r}(\MSp_{2r})$ for each $r$ satisfying
$\mu_{rs}^{*} \vartheta_{r+s} = \vartheta_{r} \times \vartheta_{s}$ for all $r,s$, and
$\vartheta_{1} |_{T^{\wedge 2}} = \Sigma_{T}^{2}1_{A}$, and

\parens{d} symplectic Thom classes theories on
$(A^{*,*},\partial,\times,1_{A})$ such that for the trivial rank $2$
bundle $\Aff^{2} \to \pt$ we have
$\thom(\Aff^{2},\omega_{2}) = \Sigma_{T}^{2}1_{A}$ in $A^{4,2}(T^{\wedge 2})$.
\end{thm}

The proof is essentially the same as that of Theorem \ref{T:a.alpha}.  The class $\vartheta_{r}$
is the \emph{tautological symplectic Thom element} of rank $2r$.

Recall
that for a commutative $T$-ring spectrum $(A,\mu,e)$ with a symplectic Thom structure
on $(A^{*,*},\partial,\times,1_{A})$ the Thom space structural map $z_{r} \colon BSp_{2r} \to \MSp_{2r}$
has the property that $z_{r}^{*} \colon A^{*,*}(\MSp_{2r}) \to A^{*,*}(BSp_{2r})$ is injective, and
that the isomorphism
\begin{equation}
\label{E:iso}
A^{*,*}(BSp_{2r}) \xla{\cong} A^{*,*}(\pt)[[b_{1},\dots,b_{r}]]
\end{equation}
derived from the symplectic Thom structure identifies the image of $z_{r}^{*}$ with the two-sided ideal
generated by $b_{r}$
(Theorems \ref{T:cohom.BSp}, \ref{T:cohom.MSp.2r} and \ref{T:cohom.MSp.2r.ideal}).

\begin{thm}
\label{T:alpha.delta}
Let $(A,\mu,e)$ be a commutative $T$-ring spectrum.
Then the assignment
$\boldsymbol{\vartheta} = (\vartheta_{1},\vartheta_{2},\vartheta_{3},\dots) \mapsto \vartheta_{1}$
gives a bijection between the sets of

\parens{$\delta$} sequences of classes
$\boldsymbol{\vartheta} = (\vartheta_{1},\vartheta_{2},\vartheta_{3},\dots)$ satisfying the
conditions of Theorem \ref{T:d.delta} and

\parens{$\alpha$} the tautological rank $2$ Thom elements $\vartheta$ of
Theorem \ref{T:a.alpha}.

The inverse bijection sends
$\vartheta \mapsto \boldsymbol{\vartheta} = (\vartheta_{1},\vartheta_{2},\vartheta_{3},\dots)$ where
$z_{r}^{*}\vartheta_{r} \in A^{4r,2r}(BSp_{2r})$ is the element corresponding to $(-1)^{r}b_{r}$
under the isomorphism \eqref{E:iso} derived from the symplectic Thom structure associated to $\vartheta$
by Theorem \ref{T:a.alpha}.
\end{thm}

\begin{proof}
Clearly the mapping $(\delta)\to (\alpha)$ is well-defined.

We will show that $(\alpha) \to (\delta)$ is well-defined.
Suppose $\vartheta$ satisfies the conditions of $(\alpha)$.
The classes $(\vartheta_{1},\vartheta_{2},\vartheta_{3},\dots)$ verify the condition
$\mu_{rs}^{*} \vartheta_{r+s} = \vartheta_{r} \times \vartheta_{s}$
because in Theorem \ref {T:cohom.BSp.BSp} the classes in the second
diagram verify $p_{r+s} \mapsto p'_{r} p''_{s}$.  The class $\vartheta_{1}$ is obtained
by the construction corresponding to the assignments $(\alpha) \to (\beta) \to (\alpha)$
of Theorem \ref{T:alpha.beta}.  So by that theorem we have $\vartheta_{1} = \vartheta$.
So we have $\vartheta_{1}|_{T^{\wedge 2}} = \Sigma_{T}^{2}1_{A}$.  Therefore
$(\alpha) \to (\delta)$ is well-defined.
In addition this shows that $(\alpha) \to (\delta) \to (\alpha)$ is the identity.

Now suppose given
$\boldsymbol{\vartheta} = (\vartheta_{1},\vartheta_{2},\vartheta_{3},\dots)$
satisfying $(\delta)$, and let
$\boldsymbol{\vartheta}' = (\vartheta'_{1},\vartheta'_{2},\vartheta'_{3},\dots)$
be the result obtained by applying $(\delta) \to (\alpha) \to (\delta)$.  We have already seen
that we have $\vartheta_{1} = \vartheta'_{1}$.  The equalities $\vartheta_{r} = \vartheta'_{r}$
follow by induction using the injectivity of the maps $m_{rs}^{*}$ of Theorem \ref{T:cohom.BSp.BSp}.
\end{proof}

For a symplectic bundle $(F,\psi)$ of rank $2r$ over $X$ the Borel classes $(\gamma)$
and the symplectic Thom classes $(\delta)$ are related by
\begin{equation}
\label{E:top.p}
b_{r}(F,\psi) = (-1)^{r} z^{*}\thom(F,\psi).
\end{equation}

\section{The universality of
{$\MSp$}{MSp}}

\begin{thm}
\label{T:cohom.MSp}
Let $(A,\mu,e)$ be a commutative $T$-ring spectrum with
a symplectic Thom classes theory on $A^{*,*}$.  Then we have isomorphisms of bigraded rings
\begin{gather*}
A^{*,*}(\MSp)
\xrightarrow{\cong}
\varprojlim
A^{*+4r,*+2r}(\MSp_{2r})
\xleftarrow{\cong}
A^{*,*}(\pt) [[b_{1},b_{2},b_{3},\dots]]^{\homog},
\\
A^{*,*}(\MSp \wedge \MSp)
\xrightarrow{\cong}
\varprojlim
A^{*+8r,*+4r}(\MSp_{2r} \wedge \MSp_{2r})
\xleftarrow{\cong}
A^{*,*}(\pt) [[b'_{1},b'_{2},\dots,b''_{1},b''_{2},\dots]]^{\homog}.
\end{gather*}
\end{thm}

\begin{proof}
By Theorem \ref{T:d.delta} the symplectic Thom classes theory has associated to it a sequence
$\boldsymbol{\vartheta} = (\vartheta_{1},\vartheta_{2},\vartheta_{3},\dots)$ of tautological
symplectic Thom classes with the property that $\vartheta_{r} \in A^{4r,2r}(\MSp_{2r})$
is the Thom class of the tautological symplectic subbundle $(\shf U_{BSp_{2r}},\phi_{BSp_{2r}})$
over $BSp_{2r} = HGr(r,\infty)$ and with $\vartheta_{1}|_{T^{\wedge 2}} = \Sigma_{T}^{2}1_{A}$.
Set also $\vartheta_{0} = 1_{A} \in A^{0,0}(\pt) = A^{0,0}(\MSp_{0})$.

By
Theorem \ref{T:inverse.lim}
the group
$A^{*,*}(\MSp)$
fits into the short exact sequence
\[
0 \to {\varprojlim}^{1}A^{\ast+4r-1,\ast+2r}(\MSp_{2r})
\to A^{\ast,\ast}(\MSp)
\to \varprojlim A^{\ast+4r,\ast+2r}(\MSp_{2r}) \to 0
\]
where the connecting maps in the tower are given by the top
line of the commutative diagram
\begin{equation*}
\vcenter{
\xymatrix @C=20pt @M=5pt {
A^{\ast+4r-4,\ast+2r-2}(\MSp_{2r-2}) &
A^{\ast+4r,\ast+2r}(\MSp_{2r-2} \wedge T^{\wedge 2}) \ar[l]_-{\Sigma^{-2}_{T}} &
A^{\ast+4r,\ast+2r}(\MSp_{2r}) \ar[l]_-{\sigma^*} \\
A^{\ast,\ast}(BSp_{2r-2}) \ar[u]_-{{-}\cup \vartheta_{r-1}}^-{\cong} &
A^{\ast,\ast}(BSp_{2r-2}) \ar[l]_-{\id}
\ar[u]_-{{-}\cup \sigma^{*}\vartheta_{r}}^-{\cong} &
A^{\ast,\ast}(BSp_{2r}) \ar[l]_-{i_{2r}^\ast} \ar[u]_-{{-}\cup \vartheta_{r}}^-{\cong}.
}}
\end{equation*}
The map $\sigma^{*}$ is the pullback along the bonding map
\[
\MSp_{2r-2} \wedge T^{\wedge 2} \xra{1 \times e_{1}} \MSp_{2r-2} \wedge \MSp_{2}
\xra{\mu_{r-1,1}} \MSp_{2r}
\]
of the symmetric $T^{\wedge 2}$-spectrum.
Thus we have $\sigma^{*}\vartheta_{r} = \vartheta_{r-1} \times \Sigma_{T}^{2}1_{A}$.
The diagram therefore commutes.  The vertical maps are isomorphisms by condition (3)
of the definition of a symplectic Thom classes theory.  The map $i_{2r}^{*}$ is the surjection
$A^{*,*}[[b_{1},\dots,b_{r}]] \onto A^{*,*}[[b_{1},\dots,b_{r-1}]]$ of \eqref{E:long.exact}.
This gives us the second isomorphism of the theorem, while the surjectivity of the connecting
maps in the inverse system gives the vanishing of the $\varprojlim^{1}$ and the first isomorphism.

The calculations for $A^{*,*}(\MSp \wedge \MSp)$ are similar.
\end{proof}

Let $\varphi \colon \MSp \to A$ be a morphism in $SH(S)$.  For each $r \geq 1$ let
$\vartheta_{r}^{\varphi} \in A^{4r,2r}(\MSp_{2r})$ be the composition
\[
\Sigma_{T}^{\infty}\MSp_{2r}(-2r) \xra{u_{2r}} \MSp \xra{\varphi} A,
\]
and let $\boldsymbol\vartheta^{\varphi} =
(\vartheta_{1}^{\varphi},\vartheta_{2}^{\varphi},\vartheta_{3}^{\varphi},\dots)$.

\begin{thm}
\label{T:daleth.delta}
Suppose $(A,\mu,e)$ is a commutative monoid in $(SH(S), \wedge, \boldsymbol 1)$.
Then the assignment $\varphi \mapsto \boldsymbol\vartheta^{\varphi} =
(\vartheta_{1}^{\varphi},\vartheta_{2}^{\varphi},\vartheta_{3}^{\varphi},\dots)$
gives a bijection between the sets of

\parens{$\varepsilon$} morphisms $\varphi \colon (\MSp,\mu^{Sp},e^{Sp}) \to (A,\mu,e)$
of commutative monoids
in $SH(S)$, and

\parens{$\delta$} sequences of classes
$\boldsymbol{\vartheta} = (\vartheta_{1},\vartheta_{2},\vartheta_{3},\dots)$ satisfying the
conditions of Theorem \ref{T:d.delta}.
\end{thm}

The proof of this theorem is substantially the same as that of Theorem \ref{T:univ.SL}.  The differences are,
first,
that the $\boldsymbol \vartheta$ comes from a unique $\varphi \colon \MSp \to A$ because the map
$A^{0,0}(\MSp) \to \varprojlim A^{4r,2r}(\MSp_{2r})$ of Theorem \ref {T:cohom.MSp}
is an isomorphism.  Second, the obstruction to $\varphi$ being a morphism of monoids vanishes
because
$A^{0,0}(\MSp \wedge \MSp) \to \varprojlim
A^{8r,4r}(\MSp_{2r} \wedge \MSp_{2r})$
is also an isomorphism.


\end{document}